\documentclass[12pt,a4paper]{article}
\def\stackunder#1#2{\lower4pt \hbox{${#2 \atop #1}$}}
\usepackage{xr}
\usepackage{url}
\usepackage{xfrac}
\usepackage{amsmath}
\usepackage{amssymb}
\usepackage{latexsym}
\usepackage{amsthm}
\usepackage{tikz}
\usetikzlibrary{matrix,calc}
\usepackage{graphicx}
\usepackage{relsize}
\usepackage{pgf}
\usepackage{tikz-cd}
\usepackage{amsfonts}
\usepackage[all]{xy}
\usepackage{etoolbox}
\usepackage{cite}
\usepackage{mathdots}
\usepackage{amscd}
\usepackage{authblk}
\usetikzlibrary{arrows,chains,matrix,positioning,scopes}
\usepackage{mathtools}
\usepackage{multicol}
\newtheoremstyle{myrem}
 {3pt}
 {3pt}
 {}
 {12pt }
 {\bfseries}
 {:}
 { }
 {}
\numberwithin{equation}{section}

\newtheorem{propn}[equation]{Proposition}

\newtheorem{corollary}[equation]{Corollary}
\newtheorem{lemma}[equation]{Lemma}

\newtheorem{thm}{Theorem}

\theoremstyle{myrem}
\newtheorem{remark}{Remark}
\appto\remark{\leftskip\parindent}

\begin{document}

\title{Syzygy Computations and the $\mathcal{D}(2)$-Problem for the Metacyclic Group $G(p,3)$}
\author{J.D.P. Evans, R. S\'{a}nchez Gal\'{a}n}
\affil{\small Dept. of Physics, Astronomy and Mathematics, University of Hertfordshire, AL10 9AB, UK,\\4i Intelligent Insights, Tecnoincubadora, Marie Curie, PCT Cartuja, 41092, Sevilla, Spain}
\date{\vspace{-5ex}}
\maketitle

\begin{abstract}
Let $p=3d+1$ be prime, let $G(p,3)=C_p\rtimes C_3$, and put
$\Lambda=\mathbb{Z}[G(p,3)]$. We study tensor products of
$\Lambda$-lattices representing syzygies and generalised syzygies of the trivial module. For $i=1,\,2,\,3$, we prove that $K(3)\otimes R(i)\cong R(i)\oplus\Lambda^{2d}$ and $K(3)\otimes K(i)\cong K(i)\oplus\Lambda^{2(p-d)}$. Thus, $K(3)$ acts as an identity on the corresponding stable classes under tensor product. We then construct explicit lattices $L$ and $Y$, representing dual stable classes associated with $K(1)$ and $K(2)$, and determine decompositions of several tensor products involving $R(1)$, $L$, and $Y$. As an application, we show that the stable isomorphism $L\sim Y^{\ast}$ is sufficient for $G(p,3)$ to have Wall's $\mathcal{D}(2)$-property and, under the same hypothesis, construct an associated six-periodic free resolution. Finally, we show that Johnson's ideal class injectivity condition implies $L\sim Y^{\ast}$, and formulate the remaining stable-isomorphism problem as a concrete integral-representation calculation.
\end{abstract}

\medskip

In what follows, we let $p$ be prime of the form $p=3d+1$, and $G(p,3)=C_p\rtimes C_3$ the metacyclic group of order $3p$, in which $C_p$ (resp. $C_3$) is the cyclic group of order $p$ (resp. $3$) generated by $x$ (resp. $y$). This can also be written as $G(p,3)=\langle x,\,y\:|\:x^p=y^3=1,\,yxy^{-1}=x^\alpha\rangle$. Next, let $\alpha\in\{2,\ldots,p-2\}$ whose residue class modulo \(p\)
has order three. be an integer whose residue class mod $p$ has order 3 in $\mathbb{F}_p^\ast$, and denote the integral group ring of $G(p,3)$ by $\Lambda=\mathbb{Z}[G(p,3)]$.

Throughout, all modules shall be treated as right modules. As we shall need to take dual modules at several points, we note that such modules are made into right modules via the canonical (anti-)involution $\bar{g}=g^{-1}$, $g\in G(p,\,3)$. In addition to duality, we shall also take the tensor product of various $\Lambda$-modules over $\mathbb{Z}$. Whenever it is clear to do so, we shall drop the subscript in $-\otimes_\mathbb{Z}-$. Next, we say two $\Lambda$-modules $M,\,M^\prime$ are \textit{stably equivalent} ($M\sim M^\prime$) when $M\oplus\Lambda^a\cong M^\prime\oplus\Lambda^b$ for some integers $a,\,b\geq 0$. The terms $r^{th}$ \textit{syzygy module} $\Omega_r(\mathbb{Z})$ of $\mathbb{Z}$ and \textit{generalised} $r^{th}$ \textit{syzygy module} $D_r(\mathbb{Z})$ of $\mathbb{Z}$ are both used in the sense of \cite{Johnson2021}. Such modules are of significant interest in the cohomology theory of finite groups, $G$. For instance, they `co-represent' cohomology in the sense that $H^r(G,N)= Hom_{\mathcal{D}er}(\Omega_r(\mathbb{Z}),\,N)$, where $\mathcal{D}er$ is the derived module category of $\Lambda$ in the sense of \cite{Johnson2003} (cf. Chapter 4).

A key result for us will be the decomposition of the augmentation ideal, $I_G\cong R(1)\oplus [y-1)$ for some module $R(1)$. This leads to what may be thought of as an $x$-strand and a $y$-strand. There exist two further modules $R(2),\,R(3)$ such that $R(k+1)\oplus[y-1)$ is a minimal element of $D_{2k+1}(\mathbb{Z})$ for each $k$. We let $K(1),\,K(2),\,K(3)$ be modules existing in exact sequences of the form $0\to K(i)\to\Lambda\to R(i)\to 0.$
\begin{thm}
	\label{Theorem_1}
	Let $X$ be either $R(i)$ or $K(i)$, where $i=1,\,2,\,3$. Then  $K(3)\otimes X\cong X\oplus\Lambda^a$ for some $a\geq 0$, i.e. $K(3)$ acts as the identity within the stable class under the tensor product.
\end{thm}

\begin{thm}
	\label{Theorem_2}
	Let $L,\,Y$ be two $\Lambda$-modules (yet to be defined) in which $L\sim K(1)^\ast$ and $Y\sim K(2)^\ast$. Then the following identities hold:
	\begin{multicols}{3}
		\begin{itemize}
			\item \small $R(1)\otimes R(1)\cong L^\ast \oplus\Lambda^{d-1}$;
			\item \small $R(1)\otimes Y\cong R(2)\oplus\Lambda^{2d}$;
			\item \small $R(1)\otimes R(2)\cong Y^\ast\oplus\Lambda^{d-1}$;
			\item \small $R(1)\otimes L\cong R(3)\oplus\Lambda^{2d}$.
			\item \small $R(1)\otimes R(3)\cong K\oplus\Lambda^{d-1}$;
			\item[{}]
		\end{itemize}
	\end{multicols}
\end{thm}

It is tempting to conjecture that $L\sim Y^\ast$. However, we do not yet know enough to make this claim. Partly this is because we do not yet know if the tree structures of the even syzygies are straight (cf. Proposition \ref{even_syzygies_straightness}), and this significantly limits our understanding. In addition to the aforementioned co-representation of cohomology, the above has an application to the $\mathcal{D}(2)$-Problem, see Section \ref{sec:the_D2_problem}.

\begin{thm}
\label{Theorem_3}
If $L\sim Y^\ast$, then $G(p,\,3)$ has the $\mathcal{D}(2)$-property.
\end{thm}

Note that Theorem \ref{Theorem_3} provides a sufficient condition consistent with that of \cite{Johnson2021}, yet significantly more computable. As we shall see, a key difficulty will be the distinction between the derivative modules $D_{2k+1}(\mathbb{Z})$ and the genuine syzygies $\Omega_{2k+1}(\mathbb{Z})$. We shall discuss at length the extent to which $R(k+1)\oplus[y-1)\in\Omega_{2k+1}(\mathbb{Z})$.
\medskip 

\noindent\textbf{Acknowledgements:} The first named author would like to warmly thank Prof. F. E. A. Johnson, not only for introducing him to this problem, but for his continued support and advice.

\section{The $\mathcal{D}(2)$-Problem}
\label{sec:the_D2_problem}

Arising in the 1960s from a broader attempt to classify compact manifolds by means of `surgery', Wall's $\mathcal{D}(2)$-Problem asks whether a 3-complex that is cohomologically two-dimensional is in fact homotopy equivalent to a 2-complex. Specifically, we let $X$ be a finite connected 3-complex with universal cover $\tilde{X}$ such that for all local coefficient systems $\mathcal{B}$ on $X$, $H^3(X,\,\mathcal{B})=H_3(\tilde{X},\,\mathbb{Z})=0$. The $\mathcal{D}(2)$-problem (cf. \cite{Wall1965}) asks whether there exists a finite 2-complex $K$ such that $X\simeq K$? Such a question has a central and paramount position in topology and has strong links with other famous problems; for example, it is a generalisation of the Eilenberg-Ganea conjecture, which states that a group with cohomological dimension 2 also has geometric dimension 2, \cite{EilenbergGanea1957}, \cite{Hambleton2019}. Moreover, Theorem 8.7 of \cite{BestvinaBrady1997} shows that either the Eilenberg-Ganea conjecture is false, or there exists a counterexample to Whitehead's conjecture that every connected subcomplex of an aspherical 2-complex is aspherical, \cite{Hambleton2019}.

The $\mathcal{D}(2)$-Problem is naturally parametrized by the fundamental group\footnote{The analogous dimension-reduction problem is settled for $n\neq 2$, so dimension two is the sole unresolved case of the general $\mathcal{D}(n)$-problem. Since the methods that apply in other dimensions do not extend to $n=2$, progress has largely proceeded by establishing the $\mathcal{D}(2)$-property for particular groups or families of groups, alongside the development of general algebraic reformulations and sufficient criteria; cf. \cite{Johnson2003}.}
 and we say that a finitely presented group $G$ possesses the $\mathcal{D}(2)$-property when every finite 3-complex $X_G$ with $\pi_1(X_G)\cong G$ which is cohomologically two-dimensional is indeed homotopy equivalent to a finite 2-complex. It is useful to note that the $\mathcal{D}(2)$-Problem is equivalent to an older problem known as the realisation problem. This asks if every algebraic 2-complex $(0\to J\to F_2\to F_1\to F_0\to\mathbb{Z}\to0)$ is geometrically realisable; that is, homotopy equivalent in the algebraic sense to a complex of the form $C_\ast(K)=(0\to\pi_2(K)\to C_2(K)\to C_1(K)\to C_0(K)\to\mathbb{Z}\to 0)$, cf. \cite{Johnson2003}. Here, $K$ is a finite 2-complex, $\pi_2(K)$ the second homotopy group of $K$, and $C_n(K)=H_n(\tilde{K}^{(n)},\,\tilde{K}^{(n-1)})$ is the group of cellular $n$-chains in the universal cover of $K$.

In what follows, we take $\pi_1(X)=G(p,3)$ but for this section we can work more generally with $G(p,\,q)$ in which $p$ is prime and $q\geq 2$ such that $q|p-1$. In \cite{Johnson2021}, two numerical conditions were found which guarantee that such a group has the $\mathcal{D}(2)$ property. The first concerns the reduced projective class group, $\widetilde{K_0}(\mathbb{Z}[C_q])$, which is required to be trivial, i.e. $\widetilde{K_0}(\mathbb{Z}[C_q])={0}$. The second is less well-known. Define $A=\mathbb{Z}(\zeta_p)^{C_q}$ to be the fixed point ring under the natural action of $C_q$ on $\mathbb{Z}(\zeta_p)$, where $\zeta_p=exp(2\pi i/p)$. The inclusion $A\hookrightarrow\mathbb{Z}(\zeta_p)$ induces a homomorphism of ideal class groups, $\nu:Cl(A)\to Cl(\mathbb{Z}(\zeta_p))$ where $\nu([P]) = [P\cdot\mathbb{Z}(\zeta_p)]$. The second condition, $Inj(p,\,q)$ is then said to hold when $\nu$ is injective, cf. \cite{Johnson2021}.

The critical object turns out to be the third syzygy $\Omega_3(\mathbb{Z})$. In \cite{Johnson2021}, using $\widetilde{K_0}(\mathbb{Z}[C_q])=0$ and $Inj(p,\,q)$, Johnson showed that $R(2)\oplus[y-1)\in\Omega_3(\mathbb{Z})$. By first showing $R(2)\oplus[y-1)$ is straight, and using a straightforward rank calculation, it follows that $R(2)\oplus[y-1)$ is the unique minimal element of $\Omega_3(\mathbb{Z})$. In particular, if this minimal element is realizable\footnote{A module $J\in\Omega_3(\mathbb{Z})$ is said to be realizable when there is a finite 2-complex $X$ with $\pi_1(X)\cong G$ and $\pi_2(X)\cong J$.} and full\footnote{A module $J\in \Omega_3(\mathbb{Z})$ is said to be full when the natural map $Aut_{\mathbb{Z}[G]}(J)\to Ker(S^J)\subset Aut_{\mathcal{D}er}(J)$ is surjective. Here, $S^J:Aut_{\mathcal{D}er}(J)\to \widetilde{K_0}(\Lambda)$ is the Swan map that sends $[\alpha]\mapsto [\lim_{\to}(\alpha,\,i)]$, cf. \cite{Johnson2021}}, then $G(p,\,q)$ has the $\mathcal{D}(2)$ property. Johnson showed that this is indeed the case, thereby providing a proof of the $\mathcal{D}(2)$-Problem when these conditions hold.

While significant theoretically, the difficulty in computing $Inj(p,\,q)$ is problematic. Furthermore, both conditions are known to fail. For instance, it was shown by Cassou-Nogu\`{e}s in \cite{Cassou-Nogues1973} that $\widetilde{K_0}(\mathbb{Z}[C_q])=0$ if and only if $2\leq q\leq 11$ or $q=13,\,14,\,17,\,19$. This condition clearly holds in our case where $q=3$. The second condition also fails, with the first such case\footnote{Discovered by R. Hill and communicated to F. E. A. Johnson privately.} occurring when $p=229$ and $q=3$. This condition is far more obscure and the reader is directed to Appendix B of \cite{Johnson2021} for more details.

The main issue is therefore one of deciding whether $R(2)\oplus[y-1)\in\Omega_3(\mathbb{Z})$. This is certainly not always the case; for instance $G(13,\,12)$ fails to have this property\footnote{Note, however, this does not prove this group fails to have the $\mathcal{D}(2)$ property. To the authors' knowledge, no such group has yet been shown to fail to have the $\mathcal{D}(2)$ property.} (see Chapter 13 of \cite{Johnson2021}). However, if this property holds then the above arguments show that this leads to the $\mathcal{D}(2)$ property. In this paper, arising from tensor product decompositions that we shall prove, we shall find another condition which guarantees this property in the case of $q=3$; namely, that of $L\sim Y^\ast$. We also show that $Inj(p,\,3)$ implies $L\sim Y^\ast$, meaning our criterion is entirely consistent with the conditions of Johnson. However, as we shall see, our condition is considerably more amenable
to direct computation. It therefore offers, at least in principle, a more concrete sufficient criterion for establishing that a group $G(p,\,3)$ has the $\mathcal{D}(2)$ property than does the verification of $Inj(p,\,3)$.

\section{Preliminaries 1: The syzygies of cyclic groups}
\label{sec:preliminaries_1}

Throughout, we work with $\mathbb{Z}[G]$-lattices, i.e. $\mathbb{Z}[G]$-modules whose underlying abelian group is finitely generated and free. When $M,\,N$ are $\mathbb{Z}[G]$-lattices of ranks $m,\,n$, respectively, the tensor product $M\otimes N$ is a $\mathbb{Z}[G]$-lattice of rank $mn$ with $G$-action given by $(\nu\otimes\omega)g=\nu g\otimes\omega g$. Working with lattices confers several advantages, notably that the dual of a short exact sequence of $\mathbb{Z}[G]$-lattices is again a short exact sequence of $\mathbb{Z}[G]$-lattices (in which the arrows are reversed). This property extends to exact sequences of finite length (see\cite{Johnson2003}, pp. 117-118). Another useful property that shall be used frequently is that within the category of $\mathbb{Z}[G]$-lattices, a module is projective if and only if it is injective, Theorem 28.5 of \cite{Johnson2003}.

We first discuss the stable syzygies of $\mathbb{Z}[C_p]$ and set $\Lambda_0=\mathbb{Z}[C_p]$. The reader is directed to \cite{Evans2021} for details. There is a free resolution of period two given by $0\to\mathbb{Z}\stackrel{\epsilon^\ast}{\to}\Lambda_0\stackrel{x-1}{\to}\Lambda_0\stackrel{\epsilon}{\to}\mathbb{Z}\to 0$, where $\epsilon$ is the augmentation map, and $\epsilon^\ast$ is its dual. We denote the augmentation ideal of $\Lambda_0$ by $I_C=ker(\epsilon)$. Recall, $I_C=span_\mathbb{Z}\{x-1,\,x^2-1,\ldots,\,x^{p-1}-1\}$ and consider the exact sequence, $0\to I_C\stackrel{i}{\to}\Lambda_0\stackrel{\epsilon}{\to}\mathbb{Z}\to 0$. Dualising gives, $0\to\mathbb{Z}\stackrel{\epsilon^\ast}{\to}\Lambda_0\stackrel{i^\ast}{\to}I_C^\ast\to 0$, where $\epsilon^\ast(1)=\Sigma=\sum^{p-1}_{r=0}x^r$ is central. We identify $I_C^\ast$ with $\Lambda_0/(\Sigma)$, which is naturally a ring. Next, we put $\nu_r=i^\ast(x^r)$ where $\nu_0=1$, and observe that we can write $\nu_r=(\nu_1)^r=\nu^r$.  If we think of $I_C^\ast$ as a $\Lambda_0$-module, then $I_C^\ast$ has a $\mathbb{Z}$-basis\footnote{It will be of later use to note that for any $1\leq r\leq p-2$, both $\{(\nu^r-1)\nu^i\:|\:0\leq i\leq p-2\}$ and $\{(\nu^r-1)^2\nu^i\:|\:0\leq i\leq p-2\}$ are linearly independent sets.} $\{1,\,\nu,\,\nu^2,\ldots,\,\nu^{p-2}\}$, in which we have $\nu^{p-1}=-1-\nu-\cdots -\nu^{p-2}$ and the action of $x$ is to multiply by $\nu$. It is well-known that $I_C\cong_{\Lambda_0} I_C^\ast$. This isomorphism fails upon introducing the $y$-action of $G(p,\,3)$.

We claim $I_C\otimes I_C\cong\mathbb{Z}\oplus\Lambda_0^{p-2}$. For $p=3d+1$, we define the following for $1\leq r\leq p-2$:
\[V(r)=span_\mathbb{Z}\{\nu^{r+k}\otimes\nu^k\:|\:0\leq k\leq p-1\}\subset I_C^\ast\otimes I_C^\ast.\]

\begin{propn}
\label{Vr-iso-Lambda}
For each $1\leq r\leq p-2$, we have $V(r)\cong\Lambda_0$.
\end{propn}

As their sum is direct, we set $V=V(1)\oplus\cdots\oplus V(p-2)$ and observe $rk_\mathbb{Z}(V)=p(p-2)$. Thus, $rk_\mathbb{Z}((I_C^\ast\otimes I_C^\ast)/V)=1$. Observe, $(I_C^\ast\otimes I_C^\ast)/V$ is torsion free and hence isomorphic to $\mathbb{Z}$. In fact, by performing elementary basis transformations to the basis $\{\nu^i\otimes\nu^j\:|\:0\leq i,\,j\leq p-2\}$ of $I_C^\ast\otimes I_C^\ast$ we obtain a new basis $\{\nu^{r+k}\otimes\nu^k\:|\:1\leq r\leq p-2,\,0\leq k\leq p-1\}\cup\{\nu^{p-2}\otimes\nu^{p-2}\}$. Then $(I_C^\ast\otimes I_C^\ast)/V$ is generated by $\natural(\nu^{p-2}\otimes\nu^{p-2})=\natural(T)$, where $\natural$ is the natural surjection $\natural:I_C^\ast\otimes I_C^\ast\to (I_C^\ast\otimes I_C^\ast)/V$ and $T=\sum_{i=0}^{p-2}\left(\sum_{j=i}^{p-2}\nu^i\otimes \nu^j\right)$. It is clear that $T\in I_C^\ast\otimes I_C^\ast$ but $T\not\in V$, and that $Tx=T$, thereby showing that $x$ acts trivially on $(I_C^\ast\otimes I_C^\ast)/V$. Thus, our isomorphism extends to one over $\Lambda_0$, i.e. $(I_C^\ast\otimes I_C^\ast)/V\cong_{\Lambda_0}\mathbb{Z}$. The claim now follows.

\section{Preliminaries 2: The Modules $R(1),\,R(2)$ and $R(3)$}
\label{sec:preliminaries_2}

The details of the following can be found in Chapters 6-9 of \cite{Johnson2021}. Recall $G(p,\,3)$ is the semidirect product $C_p\rtimes C_3$, where $C_3$ acts on $C_p$ via the natural imbedding $C_3\hookrightarrow Aut(C_p)$. Let $A=\mathbb{Z}(\zeta_p)^{C_3}$ be the fixed point ring, where $\zeta_p=exp(2\pi i/p)$. Note $A$ is a Dedekind domain containing a unique principal prime ideal $(\pi)$ which ramifies completely over $p$. We have the following subring of $M_3(A)$ consisting of quasi-triangular matrices, $\mathcal{T}_3(A,\,\pi)=\{X\in M_3(A)\:|\:x_{ij}\in(\pi)\text{ if }i>j\}$, which occurs in the following fibre square:
\begin{small}
\begin{equation*}
\begin{CD}
\Lambda	@>>>	\mathcal{T}_3(A,\,\pi)\\
@VVV	@VVV\\
\mathbb{Z}[C_3]	@>>>	\mathbb{F}_p[C_3]
\end{CD}
\end{equation*}
\end{small}
\noindent where $\mathbb{F}_p$ is the field with $p$ elements. Denote the $i^{th}$ row submodule of $\mathcal{T}_3(A,\,\pi)$ by $R(i)$ and observe $\mathcal{T}_3(A,\,\pi)$ decomposes as a direct sum of right $\Lambda$-modules, $\mathcal{T}_3(A,\,\pi)\cong R(1)\oplus R(2)\oplus R(3)$. As each $R(i)$ is monogenic, for every $i\in\{1,\,2,\,3\}$ there is an exact sequence,
\[\mathcal{S}(i)=(0\to K(i)\to\Lambda\to R(i)\to 0)\]
where $K(i)$ is the kernel of the surjection $\Lambda\to R(i)$.

To obtain a more suitable description of $R(i)$, we first define a Galois structure on a $\Lambda_0$-module $M$ to be an additive automorphism $\Theta:M\to M$ such that $\Theta^3=Id_M$ and $\Theta(m\cdot x)=\Theta(m)\cdot\theta(x)$ for all $m\in M$ and where $\theta$ is our chosen automorphism of $C_p$. By a Galois lattice, we mean a pair $(M,\,\Theta)$, where $M$ is a lattice over $\Lambda_0$ and $\Theta$ is a Galois structure on $M$. This Galois lattice becomes a right lattice over $\Lambda$ via the action $m\cdot x^ry^s=\Theta^{-s}(m\cdot x^r)$.

Let $i:\Lambda_0\hookrightarrow\Lambda$ and $j:\mathbb{Z}[C_3]\hookrightarrow\Lambda$ denote the natural inclusions. For any Galois lattice $(M,\,\Theta)$ there is an isomorphism $\mathbb{Z}[C_3]\otimes (M,\,\Theta)\cong i_\ast(M)\:\:(=M\otimes_{\mathbb{Z}[C_p]}\Lambda)$. Next, $\Lambda_0$ acquires a $\Lambda$-module structure upon extending the action of $C_p$ to one of $G(p,\,3)$ by requiring $y$ to act on the right by conjugation. Extending linearly gives a $\Lambda$-module denoted by $\overline{\Lambda_0}$. More generally, if $J\subset \Lambda_0$ is a $\Lambda_0$-submodule invariant under conjugation by $y$, then $J$ becomes a $\Lambda$-module, denoted by $\overline{J}$, in which $y$ acts in the same way. Observe $(x-1)^kI_C$ and $I_C^\ast$ are invariant under conjugation by $y$, and so we obtain $\Lambda$-modules $\overline{I_C},\,\overline{(x-1)I_C},\,\overline{I_C^\ast}$. We note the following recognition criteria for $\overline{I_C^\ast}$.

\begin{propn}
\label{Criteria_for_dual_augmentation_ideal}
Let $M$ be a $\Lambda$-lattice satisfying the following three conditions:
\begin{small}
\begin{itemize}
\item[(I)] there exists $\mu\in M$ such that $\mu\cdot y=\mu$ and $M=span_\mathbb{Z}\{\mu\cdot x^r\:|\:0\leq r\leq p-1\}$;
\item[(II)] $rk_\mathbb{Z}(M)=p-1$;
\item[(III)] $m\cdot\Sigma_x=0$ for each $m\in M$, where $\Sigma_x=\sum^{p-1}_{i=0}x^i$.
\end{itemize}
\end{small}
Then $M\cong_\Lambda \overline{I^\ast_C}$ and $\{\mu\cdot x^r\:|\:0\leq r\leq p-2\}$ is a $\mathbb{Z}$-basis for $M$.
\end{propn}

\noindent We note the following:
\begin{equation}
\label{R_i_iso_to_I_C}
R(1)\cong \overline{I_C};\:\:\:R(2)\cong \overline{(x-1)I_C};\:\:\:R(3)\cong\overline{I_C^\ast}.
\end{equation}
\begin{equation}
\label{duality_of_Ri}
R(i)^\ast\cong R(4-i).
\end{equation}
\begin{equation}
	R(k+1)\oplus[y-1)\text{ is a minimal element of }D_{2k+1}(\mathbb{Z}).
\end{equation}

\noindent It is important to note that it is not true in general that $R(k+1)\oplus[y-1)$ is a genuine syzygy of $R(1)\oplus[y-1)$. For example, in \cite{Johnson2021}, it was shown that this property fails in the case of $G(13,\,12)$. This case differs in several significant ways from that of this paper. One key difference is that $\widetilde{K_0}(\mathbb{Z}[C_{12}])\neq 0$, unlike in our case.

\begin{propn}
\label{Ri_are_straight}
For each $i$, $[R(i)]$ is straight in the sense of \cite{DyerSieradski1973}.
\end{propn}

The following proof was communicated to the first named author by F.E.A. Johnson, the details of which we now outline. The key point is the following general principle (p. 188 of \cite{Johnson2021}):
\begin{equation}
\label{iso_over_Lambda_iff_iso_over_Tq}
\text{If }\:M,\,N\:\text{ are coinduced }\:\Lambda\text{-lattices, then }\:M\cong_\Lambda N\:\:\Leftrightarrow\:\:M\cong_{\mathcal{T}_3(A,\,\pi)}N.
\end{equation}
\noindent Note that a $\Lambda$-lattice $M$ is coinduced if $M\cdot\Sigma_x=0$. This is equivalent to requiring $M$ is in the image of the functor $\rho^\ast:\mathcal{M}od_{\mathcal{T}_3(A,\,\pi)}\to\mathcal{M}od_\Lambda$ induced by the surjective ring homomorphism $\rho:\Lambda\to\mathcal{T}_3(A,\,\pi)$. Observe Corollary 51.10 in \cite{Johnson2021}:
\begin{equation}
\label{P_of_Tq_is_cancellation_semigroup}
\mathcal{P}(\mathcal{T}_3(A,\,\pi))\:\text{ is a cancellation semigroup.}
\end{equation}

\noindent From this, the following can be straightforwardly deduced:
\begin{equation}
\label{Ri_straight_over_Tq}
\text{Each stable class }\:[R(i)]\:\text{ is straight over }\:\mathcal{T}_3(A,\,\pi).
\end{equation}

To show $[R(i)]$ is straight over $\Lambda$, it suffices to consider $X\oplus\Lambda\cong R(i)\oplus\Lambda$ by Swan-Jacobinski. Working over $\Lambda_\mathbb{Q}=\mathbb{Q}[G(p,\,3)]$, we have $X_\mathbb{Q}\oplus\Lambda_\mathbb{Q}\cong_{\Lambda_\mathbb{Q}} R(i)_\mathbb{Q}\oplus\Lambda_\mathbb{Q}$. By Wedderburn's Theorem, $X_\mathbb{Q}\cong_{\Lambda_\mathbb{Q}} R(i)_\mathbb{Q}$, from which it follows that $Hom_\Lambda(X,\,\mathbb{Z}[C_3])=0$. Now let $h:X\oplus\Lambda\to R(i)\oplus\Lambda$ be a $\Lambda$-isomorphism, and consider the following diagram,
\begin{small}
\begin{center}
\begin{tikzpicture}[-stealth,
  label/.style = { font=\footnotesize }]
  \matrix (m)
    [
      matrix of math nodes,
      row sep    = 4em,
      column sep = 4em
    ]
    {
      0 & X\oplus\mathcal{T}_3(A,\,\pi) & X\oplus\Lambda & \mathbb{Z}[C_3] & 0 \\
      0 & R(i)\oplus\mathcal{T}_3(A,\,\pi) & R(i)\oplus\Lambda & \mathbb{Z}[C_3] & 0. \\
    };
  \path (m-1-3) edge node [left,  label] {$h$} (m-2-3);
  
  \path (m-1-2) edge node [above, label] {$\iota$} (m-1-3);
  \path (m-2-2) edge node [above, label] {$\iota^\prime$} (m-2-3);
  
  \path (m-1-3) edge node [above, label] {$p$} (m-1-4);
  \path (m-2-3) edge node [above, label] {$p^\prime$} (m-2-4);  
  
  \path (m-1-1) edge node [above, label] {} (m-1-2);
  \path (m-2-1) edge node [above, label] {} (m-2-2);
  \path (m-1-4) edge node [above, label] {} (m-1-5);
  \path (m-2-4) edge node [above, label] {} (m-2-5);

\end{tikzpicture}
\end{center}
\end{small}

\noindent Since $Hom_\Lambda(X,\,\mathbb{Z}[C_3])=0$, this completes to the following commutative diagram,

\begin{small}
\begin{center}
\begin{tikzpicture}[-stealth,
  label/.style = { font=\footnotesize }]
  \matrix (m)
    [
      matrix of math nodes,
      row sep    = 4em,
      column sep = 4em
    ]
    {
      0 & X\oplus\mathcal{T}_3(A,\,\pi) & X\oplus\Lambda & \mathbb{Z}[C_3] & 0 \\
      0 & R(i)\oplus\mathcal{T}_3(A,\,\pi) & R(i)\oplus\Lambda & \mathbb{Z}[C_3] & 0 \\
    };
  \path (m-1-2) edge node [left,  label] {$h_-$} (m-2-2);
  \path (m-1-3) edge node [left,  label] {$h$} (m-2-3);
  \path (m-1-4) edge node [left,  label] {$h_+$} (m-2-4);
  
  \path (m-1-2) edge node [above, label] {$\iota$} (m-1-3);
  \path (m-2-2) edge node [above, label] {$\iota^\prime$} (m-2-3);
  
  \path (m-1-3) edge node [above, label] {$p$} (m-1-4);
  \path (m-2-3) edge node [above, label] {$p^\prime$} (m-2-4);  
  
  \path (m-1-1) edge node [above, label] {} (m-1-2);
  \path (m-2-1) edge node [above, label] {} (m-2-2);
  \path (m-1-4) edge node [above, label] {} (m-1-5);
  \path (m-2-4) edge node [above, label] {} (m-2-5);

\end{tikzpicture}
\end{center}
\end{small}

\noindent in which $h_+$ is surjective. Since $\mathbb{Z}[C_3]$ is free abelian of finite rank, it follows that $h_+$ is an isomorphism, and hence so too is $h_-$. As the $\Lambda$-module structure on $R(i)\oplus\mathcal{T}_3(A,\,\pi)$ is coinduced from $\mathcal{T}_3(A,\,\pi)$, it follows that the $\Lambda$-module $X$ is coinduced from a $\mathcal{T}_3(A,\,\pi)$-module $\tilde{X}$. Consequently, $h_-$ defines a $\mathcal{T}_3(A,\,\pi)$-isomorphism, $h_-:\tilde{X}\oplus\mathcal{T}_3(A,\,\pi)\to R(i)\oplus\mathcal{T}_3(A,\,\pi)$. Proposition \ref{Ri_are_straight} now follows from (\ref{Ri_straight_over_Tq}) and (\ref{iso_over_Lambda_iff_iso_over_Tq}).

Extending this, we note Corollary II of \cite{Johnson2021}:
\begin{equation*}
\label{Straightness_of_odd_syzygies}
\text{The stability class of }R(k)\oplus [y-1)\text{ is straight for any }k.
\end{equation*}
An obvious question then becomes whether or not the same is true for the even (generalised) syzygies? At present, this is unknown, even in the smallest case $G(7,\,3)$. However,  we do note the following useful observation which follows from duality:

\begin{propn}
\label{even_syzygies_straightness}
	$\Omega_2(\mathbb{Z})$ is straight if and only if $\Omega_4(\mathbb{Z})$ is straight.
\end{propn}

\begin{proof}
	By Proposition 29.5 of \cite{Johnson2003}, $\Omega_2(\mathbb{Z})$ and $\Omega_4(\mathbb{Z})$ both have the structure of a fork, i.e. we have one or more modules at the minimal level which stabilise after adding a single copy of $\Lambda$. As such, we consider just modules at the minimal level.
	
	Suppose $\Omega_2(\mathbb{Z})$ is straight; that is, there is a unique minimal representative, call it $C$. We know $C$ occurs in an exact sequence of the form $0\to C\to \Lambda^a\to\Lambda^b\to\mathbb{Z}\to 0$ and that $C$ is also a $\Lambda$-lattice. As such, we also have the exact sequence $0\to\mathbb{Z}\to\Lambda^b\to\Lambda^a\to C^\ast\to 0$, where we have used the self-duality of $\mathbb{Z},\,\Lambda^a,\,\Lambda^b$. However, this exact sequence represents the syzygy $\Omega_{-2}(\mathbb{Z})=\Omega_4(\mathbb{Z})$, where the last equality follows from the fact $G(p,\,3)$ has cohomological period 6 (see Proposition 41.3 of \cite{Johnson2003}). Choose a minimal representative $D\in\Omega_4(\mathbb{Z})$ so that $D\oplus\Lambda\cong C^\ast\oplus\Lambda$. Dualising gives $D^\ast\oplus\Lambda\cong C^{\ast\ast}\oplus\Lambda\cong C\oplus\Lambda$, and since $\Omega_2(\mathbb{Z})$ is straight, we have $D^\ast\cong C$. Finally, $D\cong D^{\ast\ast}\cong C^\ast$, and we conclude $\Omega_4(\mathbb{Z})$ is straight, as claimed. A similar argument shows the other direction.
\end{proof}

\section{The module $K$ acts as the identity}
\label{sec:K_as_the_id}

With the preliminary legwork complete, we now show Theorem \ref{Theorem_1}; that is, we show
\begin{equation}
K(3)\otimes ?\cong ?\oplus\Lambda^r
\end{equation}
for some $r\geq 0$ and where $?=K(1),\,R(1),\,K(2),\,R(2),\,K(3),\,R(3)$. Define the module $K=[y-1,\,\Sigma_x)$, where once again $\Sigma_x=\sum^{p-1}_{i=0}x^i$.

\begin{propn}
\label{Z_basis_for_K_q_is_3}
$K$ has $\mathbb{Z}$-basis $\mathcal{E}=\{(y^i-1)x^j\:|\:1\leq i\leq 2,\,0\leq j\leq p-1\}\cup\{\Sigma_x\}$.
\end{propn}

\begin{proof}
Start with the right module $[y-1)$ which is trivially a submodule of $K$. Moreover, note that $[y-1)$ has integral basis given by $\mathcal{E}\backslash\{\Sigma_x\}$. Next, for $\Sigma_x$ observe that for $i=1,\,2$,
\[\Sigma_xy^i=\sum^{p-1}_{j=0}(y^i-1)x^j+\Sigma_x.\]
Thus, $K$ is spanned over $\mathbb{Z}$ by $\mathcal{E}$. Finally, upon performing elementary basis transformations to the standard $\mathbb{Z}$-basis of $\Lambda$, it is evident that $\mathcal{E}\cup\{x^i\:|\:1\leq i\leq p-1\}$ is a $\mathbb{Z}$-basis for $\Lambda$. It therefore follows that $\mathcal{E}$ is a $\mathbb{Z}$-basis for $K$, as required.
\end{proof}

\begin{propn}
\label{Lambda_quotient_K_is_R3}
$\Lambda/K\cong\bar{I_C^\ast}$.
\end{propn}

\begin{proof}
From the previous proof, $\mathcal{E}\cup\{x^i\:|\:1\leq i\leq p-1\}$ is a $\mathbb{Z}$-basis for $\Lambda$. Thus, $\Lambda/K$ has $\mathbb{Z}$-basis $\{\natural(x^i)\:|\:1\leq i\leq p-1\}$ where $\natural:\Lambda\to\Lambda/K$ is the identification map. It is clear that $m\Sigma_x =0$ for all $m\in\Lambda/K$ and note also that $\Lambda/K=span_\mathbb{Z}\{\natural(1)x^r\:|\:0\leq r\leq p-1\}.$ It remains to show $\natural(1)y=\natural(1)$. However, in $\Lambda$ we clearly have $1\cdot y=(y-1)+1$ and so $y$ acts trivially on $\natural(1)$ as required. The result follows from Proposition \ref{Criteria_for_dual_augmentation_ideal}.
\end{proof}

\begin{corollary}
\label{K_sim_K3}
$K\sim K(3)$.
\end{corollary}

\begin{proof}
By Proposition \ref{Lambda_quotient_K_is_R3}, $K$ arises in an exact sequence of the form $0\to K\to\Lambda\to\bar{I_C^\ast}\to 0$ and $\bar{I_C^\ast}\cong R(3)$ by (\ref{R_i_iso_to_I_C}). As we also have the exact sequence $0\to K(3)\to\Lambda\to R(3)\to 0$, the result follows from Schanuel's Lemma.
\end{proof} 

\noindent Next, we have an exact sequence, $0\to [y-1)\to K\to\mathbb{Z}\to 0$ since $x,\,y$ act trivially on $K/[y-1)$. Observe that tensoring with any of the $R(i)$ or $K(j)$ yields the exact sequence
\[0\to [y-1)\otimes ?\to K\otimes ?\to ?\to 0.\]

\noindent The following is clear:

\begin{propn}
\label{Iq_pushed_forward_by_j}
For $j:\mathbb{Z}[C_3]\hookrightarrow \Lambda$ the canonical injection, $j_\ast(I_3)=[y-1)$.
\end{propn}

\begin{propn}
\label{Ri_as_ZC3_mods}
With $j$ as above, $j^\ast(R(i))\cong\mathbb{Z}[C_3]^d$ for each $1\leq i\leq 3$.
\end{propn}

\begin{proof}
Start with the exact sequence,
\[0\to R(1)\oplus R(2)\oplus R(3)\to\Lambda\to\mathbb{Z}[C_3]\to 0.\]
Upon applying the exact functor $j^\ast(-)$ we have the split exact sequence,
\[0\to j^\ast(R(1)\oplus R(2)\oplus R(3))\to\mathbb{Z}[C_3]^p\to\mathbb{Z}[C_3]\to 0.\]
It follows that each $R(i)$ is projective as a $\mathbb{Z}[C_3]$-module. As $\widetilde{K_0}(\mathbb{Z}[C_3])=0$, any projective module is stably free. Using Swan-Jacobinski, we therefore conclude $j^\ast(R(i))\cong\mathbb{Z}[C_3]^d$.
\end{proof}

This result can be generalised provided $\widetilde{K_0}(\mathbb{Z}[C_q])=0$. Such cases have been entirely determined by Cassou-Nogu\'{e}s (cf. \cite{Cassou-Nogues1973}). Specifically, for $q$ prime, $\widetilde{K_0}(\mathbb{Z}[C_q])=0$ if and only if $q\leq 19$. For $q$ composite, $\widetilde{K_0}(\mathbb{Z}[C_q])=0$ if and only if $q=4,\,6,\,8,\,9,\,10$ or $14$. The smallest case for which this fails is $q=12$ and the reader is directed to Chapter 13 of \cite{Johnson2021} for an exposition on how this impacts computations. In particular, $j^\ast(R(k))$ need not be free.

\begin{propn}
\label{Ri_otimes_K__iso_Ri}
$K\otimes R(i)\cong R(i)\oplus\Lambda^{2d}$, where $1\leq i\leq 3$ and $d=(p-1)/3$ as before.
\end{propn}

\begin{proof}
Consider the following exact sequence,
\begin{equation}
	\label{Ri_otimes_K__iso_Ri_exact_seq}
	0\to [y-1)\otimes R(i)\to K\otimes R(i)\to R(i)\to 0.
\end{equation}
Using a special case of Curtis and Reiner's Tensor Product Theorem (p.243, \cite{CurtisReiner1990}) and Frobenius Reciprocity (Proposition 35.4, pp. 115-116, \cite{Johnson2021}), combined with Propositions \ref{Iq_pushed_forward_by_j} and \ref{Ri_as_ZC3_mods}, we have the following isomorphism:
\[j_\ast(I_3)\otimes R(i)\cong j_\ast(I_3\otimes j^\ast(R(i)))\cong j_\ast(I_3\otimes
\mathbb{Z}[C_3]^d)\cong j_\ast(\mathbb{Z}[C_3]^{2d})\cong\Lambda^{2d}.\]
\noindent Thus, \eqref{Ri_otimes_K__iso_Ri_exact_seq} splits, yielding $K\otimes R(i)\cong R(i)\oplus\Lambda^{2d}$, as required.
\end{proof}

\begin{propn}
\label{Ki_as_ZC3_mods}
$j^\ast(K(i))\cong\mathbb{Z}[C_3]^{p-d}$ for each $1\leq i\leq 3$.
\end{propn}

\begin{proof}
Start with the exact sequence, $0\to K(i)\to\Lambda\to R(i)\to 0$ and apply $j^\ast(-)$,
\[0\to j^\ast(K(i))\to\mathbb{Z}[C_3]^p\to j^\ast(R(i))\to 0.\]
By Proposition \ref{Ri_as_ZC3_mods}, $j^\ast(R(i))\cong\mathbb{Z}[C_3]^d$ and the above splits, yielding $j^\ast(K(i))\oplus\mathbb{Z}[C_3]^d\cong\mathbb{Z}[C_3]^p$, i.e. $j^\ast(K(i))$ is stably free of rank $p-d$. As $\mathbb{Z}[C_3]$ satisfies the Eichler condition (see \cite{Johnson2021}, pp. 26-27), it has stably free cancellation (SFC) and so $j^\ast(K(i))\cong\mathbb{Z}[C_3]^{p-d}$.
\end{proof}

\begin{propn}
\label{K_otimes_Ki_iso_Ki_oplus_free}
$K\otimes K(i)\cong K(i)\oplus\Lambda^{2(p-d)}$ for $1\leq i\leq 3$.
\end{propn}

\begin{proof}
Start with the following exact sequence,
\[0\to [y-1)\otimes K(i)\to K\otimes K(i)\to K(i)\to 0.\]
It suffices to consider $[y-1)\otimes K(i)$. Adopting the same strategy as in Proposition \ref{Ri_otimes_K__iso_Ri} shows,
\[j_\ast(I_3)\otimes K(i)\cong j_\ast(I_3\otimes j^\ast(K(i)))\cong j_\ast(I_3\otimes\mathbb{Z}[C_3]^{(p-d)})\cong j_\ast(\mathbb{Z}[C_3]^{2(p-d)})\cong \Lambda^{2(p-d)}.\]

\noindent Upon replacing $[y-1)\otimes K(i)$, we note the above exact sequence splits.
\end{proof}

\noindent Theorem \ref{Theorem_1} therefore follows from Propositions \ref{Ri_otimes_K__iso_Ri} and \ref{K_otimes_Ki_iso_Ki_oplus_free}, as we now show. First, note the following straightforward result:

\begin{lemma}
\label{lem:tensor_stable_isomorphism}
Let $M,\,N,\,X$ be $\Lambda$-lattices such that $n=rk_{\mathbb{Z}}(X)$. If $M\oplus\Lambda\cong N\oplus\Lambda$, then $(M\otimes_{\mathbb Z}X)\oplus\Lambda^{n}\cong(N\otimes_{\mathbb Z}X)\oplus\Lambda^{n}$.
\end{lemma}

\begin{proof}[Proof of Theorem~A]
By Corollary \ref{K_sim_K3} and the Swan-Jacobinski theorem, we have $K\oplus\Lambda\cong K(3)\oplus\Lambda$. We first consider $R(i)$ and set $r_i:=rk_{\mathbb{Z}}R(i)$. Lemma \ref{lem:tensor_stable_isomorphism}, together with
Proposition \ref{Ri_otimes_K__iso_Ri}, gives
\[
    \bigl(K(3)\otimes R(i)\bigr)\oplus\Lambda^{r_i}
    \cong
    \bigl(K\otimes R(i)\bigr)\oplus\Lambda^{r_i}
    \cong
    R(i)\oplus\Lambda^{2d+r_i}.
\]
Thus $K(3)\otimes R(i)$ and $R(i)\oplus\Lambda^{2d}$ belong to the same stable class and have the same $\mathbb{Z}$-rank. Since the stable class $[R(i)]$ is straight (Proposition \ref{Ri_are_straight}), there is a unique isomorphism class at
each level. Consequently, $K(3)\otimes R(i) \cong R(i)\oplus\Lambda^{2d}$, as required.

We next consider $K(i)$, setting $n_i:=rk_{\mathbb{Z}}K(i)$ and $b=2(p-d)$. By Lemma \ref{lem:tensor_stable_isomorphism}, 
Proposition \ref{K_otimes_Ki_iso_Ki_oplus_free} and Corollary \ref{K_sim_K3},
\[
    \bigl(K(3)\otimes K(i)\bigr)\oplus\Lambda^{n_i}
    \cong
    \bigl(K\otimes K(i)\bigr)\oplus\Lambda^{n_i}
    \cong
    K(i)\oplus\Lambda^{b+n_i}.
\]
It follows that $K(3)\otimes K(i)$ and $K(i)\oplus\Lambda^{b}$ belong to the same stable class and
have the same $\mathbb{Z}$-rank. Note that the stable class $[K(i)]$ is a fork (see Chapter 3 of \cite{Johnson2003}), and hence contains a
unique isomorphism class at every level apart from the bottom one. Both modules above lie above this level since $b=2(p-d)>0$, and so $K(3)\otimes K(i) \cong K(i)\oplus\Lambda^{2(p-d)}$, as required.
\end{proof}

Finally, note that in \cite{Johnson2021} (pp. 230-232), Johnson shows the construction of, amongst other things, an exact sequence of the form $0\to R(1)\to\Lambda\to K(3)\to 0$. Not only does this demonstrate the non-obvious fact that $K(3)$ is monogenic, but also the following straightforward consequence:
\begin{equation}
\label{K_stably_self_dual}
K\sim K^\ast.
\end{equation}

\section{$R(1)\otimes R(1)\sim K(1)$}
\label{sec:theorem_B_part_1}

In this section, we intend to show $\bar{I^\ast_C}\otimes\bar{I^\ast_C}\cong L\oplus\Lambda^{d-1}$ for some (yet to be defined) module $L\sim K(1)^\ast$. An obvious question is whether $K(1)^\ast\cong L$. At present it is unclear whether this is true with the key issue being whether or not $[K(1)^\ast]$ is straight.

To proceed, we first construct the free part of $\bar{I_C^\ast}\otimes\bar{I_C^\ast}$. Recall, $yxy^{-1}=x^\alpha$ and we take the Galois action, $\nu^i\cdot y=\nu^{\alpha^2i}$ and $\nu^i\cdot y^{-1}=\nu^{\alpha i}$. We first show the following:
\begin{equation}
\label{D1_initial}
V(r)+V(\alpha r)+V(\alpha^2 r)\cong\Lambda\text{ for }1\leq r\leq p-2
\end{equation}
where $V(r)=span_\mathbb{Z}\{\nu^{r+k}\otimes\nu^k\:|\:0\leq k\leq p-1\}$. Observe the slight abuse of notation whereby we are taking $\alpha r\:(mod\:p)$ and $\alpha^2 r\:(mod\:p)$ when defining $V(\alpha r)$ and $V(\alpha^2 r)$, respectively. Note, there will be two $V(i)$ `left over' which shall be discussed at the end of this section.

\begin{propn}
\label{representation_of_V0_x_action}
Set $V_r:=V(r)+V(\alpha r)+V(\alpha^2 r)$. The representation of $V_r$ with respect to the $x$-action on the defining basis of $V_r$ is the $3p\times 3p$ block matrix given by
\begin{equation*}
\rho_{V_r}(x^{-1})=
\begin{pmatrix}
\Psi&0&0\\
0&\Psi&0\\
0&0&\Psi
\end{pmatrix}
,
\qquad\text{ where }\qquad
\Psi_{ij}=
\begin{cases}
1,&i=1,\,j=p;\\
1,&j=i-1,\,2\leq i\leq p;\\
0,&\text{o/w.}
\end{cases}
\end{equation*}
\end{propn}

\begin{proof}
Denote the $\mathbb{Z}$-bases of $V(r),\, V(\alpha r),\,V(\alpha^2 r)$ by $\{e_i\:|\:1\leq i\leq p\},\, \{e_{i+p}\:|\:1\leq i\leq p\}$ and $\{e_{i+2p}\:|\:1\leq i\leq p\}$ respectively, where $e_i=\nu^{r+i-1}\otimes \nu^{i-1}$, $e_{p+i}=\nu^{\alpha r+i-1}\otimes\nu^{i-1}$ and $e_{2p+i}=\nu^{\alpha^2 r+i-1}\otimes\nu^{i-1}$. It is clear that $e_i\cdot x=e_{i+1}$ for $1\leq i\leq p-1$, and $e_p\cdot x=e_1$. Likewise, $e_{p+i}\cdot x=e_{p+i+1}$ for $1\leq i\leq p-1$, and $e_{2p}\cdot x=e_{p+1}$. Finally, $e_{2p+i}\cdot x=e_{2p+i+1}$ for $1\leq i\leq p-1$, and $e_{3p}\cdot x=e_{2p+1}$. The result now follows.
\end{proof}

\begin{propn}
\label{representation_of_Vr}
The representation of $V_r$ with respect to the action of $y^2$ is given by
\[
\rho_{V_r}(y)=
\begin{pmatrix}
0&0&\Phi^T\\
\Phi^T&0&0\\
0&\Phi^T&0
\end{pmatrix}
\qquad\text{ where }\qquad
\Phi^T_{ij}=
\begin{cases}
1,&i=(j-1)\alpha+1\:(mod\:p);\\
0,&\text{o/w.}
\end{cases}
\]
\end{propn}

\begin{proof}
With the $e_i,\,e_{p+i},\,e_{2p+i}$ as before, observe $e_i y^{-1}=\nu^{\alpha r+\alpha(i-1)}\otimes\nu^{\alpha(i-1)}$ where we have $\alpha r+\alpha(i-1)-\alpha(i-1)=\alpha r$ and $0\leq \alpha(i-1)\:(mod\:p)\leq p-1$. Thus, $e_i y^{-1}=e_{p+\alpha(i-1)+1}$. This gives us a $p\times p$ block which we label $\Phi^T$. Likewise, $e_{p+j}y^{-1}=e_{2p+\alpha(j-1)+1}$ and $e_{2p+k}y^{-1}=e_{\alpha(k-1)+1}$, where the last equality follows from $\alpha^3\equiv 1\:(mod\:p)$. It is straightforward to see that we get two more (identical) $p\times p$ blocks which are arranged as claimed.
\end{proof}

\begin{propn}
\label{regular_rep_x}
The regular representation of the $x$-action is given by $\rho_{reg}(x^{-1})=\rho_{V_r}(x^{-1})$.
\end{propn}

\begin{proof}
Set $f_i=x^{i-1}$, $f_{p+j}=yx^{j-1}$, $f_{2p+k}=y^2x^{k-1}$ where $1\leq i,\,j,\,k\leq p$. Then $f_i\cdot x=f_{i+1}$ for $1\leq i\leq p-1$, and $f_p\cdot x=f_1$. Similarly, $f_{p+j}\cdot x=f_{p+j+1}$ for $1\leq i\leq p-1$, and $f_{2p}\cdot x=f_{p+1}$. Finally, $f_{2p+k}\cdot x=f_{2p+k+1}$ for $1\leq i\leq p-1$, and $f_{3p}\cdot x=f_{2p+1}$.
\end{proof}

\begin{propn}
\label{regular_rep}
The regular representation of the $y^{-1}$-action is given by
\begin{equation*}
\rho_{reg}(y)=
\begin{pmatrix}
0&\Phi^T&0\\
0&0&\Phi^T\\
\Phi^T&0&0
\end{pmatrix}
\end{equation*}
\end{propn}

\begin{proof}
With $f_i,\,f_{p+j},\,f_{2p+k}$ as above, it is clear that $f_i y^{-1}=x^{i-1}y^2=y^2 x^{\alpha(i-1)}=f_{2p+\alpha(i-1)+1}$. Letting $i$ vary over $\{1,\ldots,\,p\}$ gives us the block $\Phi^T$. A similar argument applies to the action of $y^2$ on $f_{p+j}$ and $f_{2p+k}$, each generating the $p\times p$ block $\Phi^T$, arranged as claimed.
\end{proof}

\begin{propn}
$V(r)+V(\alpha r)+V(\alpha^2 r)\cong\Lambda$ for $1\leq r\leq p-2$.
\end{propn}

\begin{proof}
It suffices to show there exists an invertible $(3p)\times (3p)$ matrix $X$ such that $\rho_{reg}(g)X=X\rho_{V_r}(g)$ for all $g\in G(p,\,3)$. For the $p\times p$ identity matrix $I_p$ , set
\begin{equation*}
X=
\begin{pmatrix}
0&0&I_p\\
0&I_p&0\\
I_p&0&0
\end{pmatrix}
\end{equation*}
and observe
\[
\rho_{reg}(x^{-1})X=
\begin{pmatrix}
0&0&\Psi\\
0&\Psi&0\\
\Psi&0&0
\end{pmatrix}
=X\rho_{V_r}(x^{-1}).\]
and
\[
\rho_{reg}(y)X=
\begin{pmatrix}
0&\Phi^T&0\\
\Phi^T&0&0\\
0&0&\Phi^T
\end{pmatrix}
=X\rho_{V_r}(y).\]
Finally, $X$ is clearly invertible as a series of row permutations transform $X$ into $I_{3p}$.
\end{proof}

Since $3|p-1$, we get $d-1$ copies of $\Lambda$ from $V(1)+\cdots +V(p-2)$ with two of the $V(r)$'s left over (call them $V(s),\,V(t)$). These two $V(r)$'s along with $T$ will represent a module stably isomorphic to $K(1)^\ast$. With this in mind, write $V^\prime$ for the sum of the $V(i)$ \textit{without} $V(s)$ and $V(t)$, i.e. $V^\prime\cong\Lambda^{d-1}$. Now let $\eta:\bar{I^\ast_C}\otimes\bar{I^\ast_C}\to (\bar{I^\ast_C}\otimes\bar{I^\ast_C})/V^\prime$ be the natural map and write this quotient\footnote{Note the abuse of notation whereby we write $T$ for the $\mathbb{Z}$-span of the element $T$.} as $\eta(T+V(s)+V(t))$. In particular, we have the following:

\begin{propn}
\label{I_otimes_I_iso_L_free}
$\bar{I^\ast_C}\otimes\bar{I^\ast_C}\cong \eta(T+V(\alpha+1)+V(\alpha^2+1))\oplus\Lambda^{d-1}.$
\end{propn}

\begin{proof}
Since $p|\alpha^3-1$ it follows that $p|\alpha^2+\alpha+1$. It is therefore clear that the only possible choices for the aforementioned $s$ and $t$ are $\alpha+1$ and $\alpha^2+1\equiv p-\alpha$, i.e. $T+ V(\alpha+1)+ V(\alpha^2+1)$ is left over when we introduce the $y$-action to $T\oplus V$ (thought of as a $\Lambda_0$-module).

Now, $(\bar{I^\ast_C}\otimes\bar{I^\ast_C})/V^\prime$ is torsion free (Proposition 3.1 of \cite{Johnson2015}) and hence a lattice. We therefore have the following split short exact sequence, $0\to V^\prime\to \bar{I^\ast_C}\otimes\bar{I^\ast_C}\to\eta(T+V(\alpha+1)+V(\alpha^2+1))\to 0$ since $V^\prime\cong\Lambda^{d-1}$. The result now follows
\end{proof}

\begin{propn}
\label{K(1)_sim_L_dual}
Set $L=\eta(T+ V(\alpha+1)+V(\alpha^2+1))$. We then have $K(1)\sim L^\ast$.
\end{propn}

\begin{proof}
Note, $I_G\cong\bar{I}_C\oplus[y-1)$ (pp. 229-230 of \cite{Johnson2021}). Recall, $[y-1)=j_\ast(I_3)$ is self dual. Next, consider the exact sequence, $0\to\bar{I}_C\oplus[y-1)\to\Lambda\to\mathbb{Z}\to 0$
and apply $-\otimes\bar{I}_C$. In the proof of Proposition \ref{Ri_otimes_K__iso_Ri} we have shown $\bar{I}_C\otimes[y-1)\cong\Lambda^{2d}$. Using this along with Frobenius Reciprocity on the middle term yields the exact sequence,
\begin{equation}
\label{IC_IC_intermediate_exact_sequence}
	0\to(\bar{I}_C\otimes\bar{I}_C)\oplus\Lambda^{2d}\stackrel{\iota}{\to}\Lambda^{3d}\to\bar{I}_C\to 0.
\end{equation}

By the dual of Proposition \ref{I_otimes_I_iso_L_free}, $\bar{I_C}\otimes\bar{I_C}\cong L^\ast\oplus\Lambda^{d-1}$. Substituting this back into (\ref{IC_IC_intermediate_exact_sequence}) we obtain $0\to L^\ast\oplus\Lambda^{3d-1}\stackrel{\iota}{\to}\Lambda^{3d}\to\bar{I}_C\to 0$. This can be transformed into $0\to L^\ast\to\Lambda^{3d}/\iota(\Lambda^{3d-1})\to\bar{I}_C\to 0$. By Johnson's `destabilization theorem' (see Propositions 22.1 and 34.2 of \cite{Johnson2021}, pp. 85, 113), $\Lambda^{3d}/\iota(\Lambda^{3d-1})$ is projective. We can therefore construct the split exact sequence $0\to\Lambda^{3d-1}\stackrel{\iota}{\to}\Lambda^{3d}\to\Lambda^{3d}/\iota(\Lambda^{3d-1})\to 0$. As the sequence splits, $\Lambda^{3d}/\iota(\Lambda^{3d-1})$ is stably free of rank $1$. However, $\Lambda$ has SFC and so there are no nontrivial stably free modules. We therefore obtain the exact sequence $0\to L^\ast\to\Lambda\to\bar{I}_C\to 0$. Since $\bar{I}_C\cong R(1)$, the result now follows from Schanuel's Lemma, i.e. $L^\ast\sim K(1)$.
\end{proof}

\noindent To summarise, we have shown the following:

\begin{propn}
\label{IC_otimes_IC_iso_L_plus_free}
With, $L,\,L^\ast$ as above, $\bar{I}_C\otimes\bar{I}_C\cong L^\ast\oplus\Lambda^{d-1}$ and $\bar{I^\ast_C}\otimes\bar{I^\ast_C}\cong L\oplus\Lambda^{d-1}$.
\end{propn}

\noindent From this, we obtain the following result: 

\begin{propn}
	$K(1)\oplus j_\ast(\mathbb{Z})\in\Omega_2(\mathbb{Z})$.
\end{propn}

\begin{proof}
Start with the exact sequence $0\to I_G\to\Lambda\to\mathbb{Z}\to0$ and alter it to $0\to I_G\otimes I_G\to\Lambda^{3p-1}\to I_G\to 0.$ Since $I_G\cong\bar{I_C}\oplus[y-1)$ we use Proposition \ref{IC_otimes_IC_iso_L_plus_free} and the fact that $\bar{I_C}\otimes[y-1)\cong\Lambda^{2d}$ to obtain the exact sequence $0\to L^\ast\oplus X\oplus\Lambda^{5d-1}\to\Lambda^{3p-1}\to I_G\to 0$, where $X=[y-1)\otimes[y-1)$. Using Johnson's `destabilization theorem':
\[
0\to L^\ast\oplus X\to\Lambda^{4d+3}\to I_G\to 0.\tag{$\dagger$}
\]
Using the special case of Curtis and Reiner's Tensor Product Theorem (p.243, \cite{CurtisReiner1990}) along with Frobenius Reciprocity once more, we observe
\[
	j_\ast(I_3)\otimes([y-1)\oplus\Lambda)
	\cong j_\ast(I_3\otimes( j^\ast([y-1))\oplus\mathbb{Z}[C_3]^{p}))\tag{$\ddagger$}
\]
Next, consider the exact sequence $0\to[y-1)\to K\to\mathbb{Z}\to0$ and apply the exact functor $j^\ast(-)$ so that we have $0\to j^\ast([y-1))\to\mathbb{Z}[C_3]^{(p-d)}\to\mathbb{Z}\to 0$. Since we also have $0\to I_3\to\mathbb{Z}[C_3]\to\mathbb{Z}\to 0$ we apply Schanuel's Lemma to conclude that $j^\ast([y-1))\oplus\mathbb{Z}[C_3]^p\cong I_3\oplus\mathbb{Z}[C_3]^{(2p-d-1)}$. Plugging this into $\ddagger$, we have
\[
j_\ast(I_3)\otimes([y-1)\oplus\Lambda)\cong j_\ast(I_3\otimes (I_3\oplus\mathbb{Z}[C_3]^{(2p-d-1)}))\cong j_\ast(\mathbb{Z}\oplus\mathbb{Z}[C_3]^{(4p-2d-1)})\cong j_\ast(\mathbb{Z})\oplus\Lambda^{10d+3}.
\]
Finally, modify $\dagger$ into $0\to L^\ast\oplus X\oplus\Lambda^{2p}\to\Lambda^{4d+2p+3}\to I_G\to 0$ and plug in the above so that we get, after applying Johnson's `destabilization theorem' one more time, $0\to L^\ast\oplus j_\ast(\mathbb{Z})\to \Lambda^2\to I_G\to 0.$ In other words, $L^\ast\oplus j_\ast(\mathbb{Z})\in\Omega_1(I_G)=\Omega_2(\mathbb{Z})$, as required.
\end{proof}

\section{$R(1)\otimes R(3)\sim K$}
\label{sec:theorem_B_part_5}

With the $\Lambda$-modules $\bar{I_C^\ast}$, $\bar{I}_C=\overline{(\zeta_p-1)I_C^\ast}$ and $K$ as before, we aim to show
\begin{equation}
\bar{I}_C\otimes\bar{I_C^\ast}\cong K\oplus\Lambda^{d-1}.
\end{equation}
Observe $\{(\nu-1)\nu^i\otimes\nu^j\:|\:0\leq i,\,j\leq p-2\}$ forms a $\mathbb{Z}$-basis for $\bar{I}_C\otimes\bar{I_C^\ast}$ and the $x,\,y$-actions are as explained in Section \ref{sec:theorem_B_part_1}. We proceed in the following three stages:

\begin{itemize}
\item[] Stage 1: Find a description for $\bar{I}_C\otimes \bar{I}_C^\ast$, part of which is isomorphic to $K$;
\item[] Stage 2: Of what is left, demonstrate this is free;
\item[] Stage 3: Deduce that $\bar{I}_C\otimes\bar{I_C^\ast}\cong K\oplus\Lambda^{d-1}$.
\end{itemize}

\noindent \textbf{Stage 1:} Define the following modules for $1\leq r\leq p-2$:
\begin{equation}
U(r)=span_\mathbb{Z}\{(\nu-1)\nu^{r+k}\otimes\nu^k\:|\:0\leq k\leq p-1\}\subset I_C\otimes I_C^\ast.
\end{equation}

\noindent Using essentially the same arguments as Chapter 3 of \cite{Evans2018}, we can show the following:
\medskip

\noindent\textbf{Fact 1:} For each $r$, $\{(\nu-1)\nu^{r+k}\otimes\nu^k\:|\:0\leq k\leq p-1\}$ is a $\mathbb{Z}$-basis for $U(r)$, i.e. $U(r)\cong_{\Lambda_0}\Lambda_0$.

\noindent\textbf{Fact 2:} $U(r)\cap [U(1)+\cdots+U(r-1)+U(r+1)+\cdots+U(p-2)]=\{0\}$.

\noindent\textbf{Fact 3:} Set $U:=U(1)\oplus\cdots\oplus U(p-2)$. Then $I_C\otimes I_C^\ast/U$ is a rank 1 lattice generated by the image of $S=\sum^{p-2}_{i=0}(\sum^{p-2}_{j=i}(\nu-1)\nu^i\otimes\nu^j)$.

\noindent\textbf{Fact 4:} $I_C\otimes I_C^\ast\cong_{\Lambda_0}[\natural(S))\oplus U$, where $\natural: I_C\otimes I_C^\ast\to I_C\otimes I_C^\ast/U$.
\medskip

\noindent Using $\nu^i-1=\sum^{i-1}_{j=0}(\nu-1)\nu^j$, we rewrite $S=(\nu-1)\otimes 1+(\nu^2-1)\otimes\nu+\cdots+(\nu^i-1)\otimes\nu^{i-1}+\cdots+(\nu^{p-1}-1)\otimes\nu^{p-2}$. Note, $Sx=S$.

Before introducing the $y$-action, we first make several alterations to the $\mathbb{Z}$-basis of $[\natural(S))\oplus U$. First, we consider the elements of $U$ of the form $-\otimes\nu^i$. The left hand side of these looks like $(\nu-1)\nu^r$ where $i+1\leq r\leq i+p-2$. In particular, by summing the first $r$ of these elements, we get $(\nu-1)(\nu^{i+r}+\cdots+\nu^{i+2}+\nu^{i+1})=(\nu-1)(\nu^{r-1}+\cdots+\nu+1)\nu^{i+1}=(\nu^r-1)\nu^{i+1}$. By elementary basis change, we can therefore replace $U(r)$ by
\[U(r)^\prime=span_\mathbb{Z}\{(\nu^r-1)\nu^{i+1}\otimes\nu^i\:|\:0\leq i\leq p-1\}.\]
Trivially, $U(1)^\prime =U(1)$. We now replace\footnote{Notice we have dropped $\natural$ from $\natural(S)$, and have written $S$ instead of $[S)$. This is to simplify notation.} $S\oplus U$ by $S\oplus U^\prime$, where $U^\prime=U(1)^\prime\oplus\cdots\oplus U(p-2)^\prime$. By introducing the $y$-action, we will show $S+U(\alpha-1)^\prime+U(\alpha^2-1)^\prime\cong K$. Set $U_\alpha:=U(\alpha-1)^\prime+U(\alpha^2-1)^\prime$ which, as we will see, is isomorphic to $[y-1)$. The following is clear.

\begin{propn}
\label{U_alpha_x_action}
\small{$\rho_{U_\alpha}(x^{-1})=\begin{pmatrix}\Psi&0\\0&\Psi\end{pmatrix}$}, where $\Psi$ is the $p\times p$ block defined in Proposition \ref{representation_of_V0_x_action}.
\end{propn}

\begin{propn}
\label{U_alpha_y_action}
\small{$\rho_{U_\alpha}(y)=\begin{pmatrix}-\Phi^T&-\Phi^T\\\Phi^T&0\end{pmatrix}$}, where $\Phi^T$ is the $p\times p$ block of Proposition \ref{representation_of_Vr}.
\end{propn}

\begin{proof}
Label $e_i=(\nu^{\alpha-1}-1)\nu^i\otimes\nu^{i-1}$ and $e_{p+j}=(\nu^{\alpha^2-1}-1)\nu^j\otimes\nu^{j-1}$ for $1\leq i,\,j\leq p$ and consider the action of $y^2$. We have the following:
\begin{eqnarray*}
[(\nu^{\alpha-1}-1)\nu^{i}\otimes \nu^{i-1}]y^2=&(\nu^{\alpha^2-\alpha}-1)\nu^{\alpha i}\otimes \nu^{\alpha(i-1)}\\
=&(\nu^{\alpha^2}-\nu+\nu-\nu^\alpha)\nu^{\alpha(i-1)}\otimes \nu^{\alpha(i-1)}\\
=&(\nu^{\alpha^2-1}-1)\nu^{\alpha(i-1)+1}\otimes \nu^{\alpha(i-1)}-(\nu^{\alpha-1}-1)\nu^{\alpha(i-1)+1}\otimes \nu^{\alpha(i-1)}\\
=&e_{p+\alpha(i-1)+1}-e_{\alpha(i-1)+1}.
\end{eqnarray*}
For the elements $e_{p+i}$, we use $\alpha^3=1\:(mod\:p)$:
\[
[(\nu^{\alpha^2-1}-1)\nu^{i}\otimes\nu^{i-1}]y^2=(\nu^{1-\alpha}-1)\nu^{\alpha i}\otimes \nu^{\alpha(i-1)}
=(\nu-\nu^\alpha)\nu^{\alpha(i-1)}\otimes \nu^{\alpha(i-1)}=-e_{\alpha(i-1)+1}.
\]
\end{proof}

\begin{propn}
\label{y_applied_to_S}
$S\cdot y^2=S+\sum^{p-1}_{i=0}(\nu^{\alpha-1}-1)\nu^{i+1}\otimes\nu^i$.
\end{propn}

\begin{proof}
Take $S=\sum^{p-1}_{i=1}(\nu^{i}-1)\otimes \nu^{i-1}$ and note $\nu^{\alpha(i-1)}=\nu^{p-1}$ when $i=p-\alpha^2+1$. Then,
\begin{eqnarray*}
Sy^2=&\sum_{i\neq p-\alpha^2+1}(\nu^{\alpha i}-1)\otimes\nu^{\alpha(i-1)}+(\nu^{\alpha-1}-1)\otimes\nu^{p-1}\\
=&\scalebox{0.9}{$\sum_{i\neq p-\alpha^2+1}(\nu^{\alpha i}-\nu^{\alpha(i-1)+1})\otimes\nu^{\alpha(i-1)}+\sum_{i\neq p-\alpha^2+1}(\nu^{\alpha(i-1)+1}-1)\otimes\nu^{\alpha(i-1)}+(\nu^{\alpha-1}-1)\otimes\nu^{p-1}$}\\
=&\underbrace{\scalebox{0.9}{$\sum_{i\neq p-\alpha^2+1}(\nu^{\alpha-1}-1)\nu^{\alpha(i-1)+1}\otimes\nu^{\alpha(i-1)}+(\nu^{\alpha-1}-1)\otimes\nu^{p-1}$}}_{\ast}+\underbrace{\scalebox{0.9}{$\sum_{i\neq p-\alpha^2+1}(\nu^{\alpha(i-1)+1}-1)\otimes\nu^{\alpha(i-1)}$}}_{\ast\ast}.
\end{eqnarray*}

\noindent First, we observe that $\nu^{\alpha(i-1)}=\nu^{\alpha(j-1)}$ if and only if $i=j$. Next, we note $\ast$ is simply a sum of elements in $U(\alpha-1)^\prime$ and that $\nu^{\alpha(i-1)}\neq\nu^{p-\alpha}$ for any $i\in\{1,\ldots,\,p-1\}$. Moving to $\ast\ast$, we observe $\ast\ast=S-(\nu^{p-\alpha+1}-1)\otimes\nu^{p-\alpha}=S+(\nu^{\alpha-1}-1)\nu^{p-\alpha+1}\otimes\nu^{p-\alpha}$. As this last term is the `missing' term from the basis of $U(\alpha-1)^\prime$, the result is shown.
\end{proof}

\noindent Using $Sx=S$ and Propositions \ref{U_alpha_x_action}, \ref{U_alpha_y_action} and \ref{y_applied_to_S} we have the following two results:

\begin{propn}
Set $K^\prime=S+U_\alpha$; then
\begin{equation*}
\small 
\rho_{K^\prime}(x^{-1})=
\begin{pmatrix}
\Psi&0_{p\times p}&0_{p\times 1}\\
0_{p\times p}&\Psi&0_{p\times 1}\\
0_{1\times p}&0_{1\times p}&1_{1\times 1}
\end{pmatrix}
.
\end{equation*}
\end{propn}

\begin{propn}
\begin{equation*}
\small
\rho_{K^\prime}(y)=
\begin{pmatrix}
-\Phi^T&-\Phi^T&1_{p\times 1}\\
\Phi^T&0_{p\times p}&0_{p\times 1}\\
0_{1\times p}&0_{1\times p}&1_{1\times 1}
\end{pmatrix}
\end{equation*}
\end{propn}

Denote the $\mathbb{Z}$-basis of $[y-1)$ by $\{f_i,\,f_{p+j}\:|\:1\leq i,\,j\leq p\}$, where $f_i=(y-1)x^{i-1}$ and $f_{p+j}=(y^2-1)x^{j-1}$. Recall also that $\Sigma_x y^2=\Sigma_x(y^2-1)+\Sigma_x$. Then the following are clear:

\begin{propn}
$\rho_{[y-1)}(x^{-1})=\begin{pmatrix}\Psi&0\\0&\Psi\end{pmatrix}$ where $\Psi$ is the $p\times p$ block given above.
\end{propn}

\begin{corollary}
$\small 
\rho_{K}(x^{-1})=
\begin{pmatrix}
\Psi&0_{p\times p}&0_{p\times 1}\\
0_{p\times p}&\Psi&0_{p\times 1}\\
0_{1\times p}&0_{1\times p}&1_{1\times 1}
\end{pmatrix}
.
$
\end{corollary}

\begin{propn}
$\small \rho_{[y-1)}(y)=\begin{pmatrix}0&\Phi^T\\-\Phi^T&-\Phi^T\end{pmatrix}$ where $\Phi^T$ is the $p\times p$ block given above.
\end{propn}

\begin{corollary}
$\small 
\rho_{K}(y)=
\begin{pmatrix}
0_{p\times p}&\Phi^T&0_{p\times 1}\\
-\Phi^T&-\Phi^T&1_{p\times 1}\\
0_{1\times p}&0_{1\times p}&1_{1\times 1}
\end{pmatrix}
.
$
\end{corollary}

\begin{propn}
\label{S_plus_U_alpha_is_K}
$S+U_\alpha\cong K$.
\end{propn}

\begin{proof}
Setting the following to be $X$ which is clearly invertible over the integers,
\[
\small
X=
\begin{pmatrix}
	0_{p\times p}&I_p&0_{p\times 1}\\
	I_p&0_{p\times p}&0_{p\times 1}\\
	0_{1\times p}&0_{1\times p}&1_{1\times 1}
\end{pmatrix},\]
it is now a straightforward observation that $\rho_{K^\prime}(g)X=X\rho_K(g)$ for all $g\in G(p,\,3)$.
\end{proof}

We now show $(S+U^\prime)/(S+U_\alpha)$ is free. Take the natural map $\eta:S+U^\prime\to(S+U^\prime)/(S+U_\alpha)$ and set $U_r^\prime=U(r)^\prime+U(\alpha(r+1)-1)^\prime+U(\alpha^2(r+1)-1)^\prime$ for $r\neq \alpha-1,\,\alpha^2-1$. We claim $\eta(U_r^\prime)\cong\Lambda$ for each $r\neq \alpha-1,\,\alpha^2-1$, where $\eta(U_r^\prime)$ represents the image of $U_r^\prime$ under $\eta$.

\begin{propn}
Set $U_r=\eta(U_r^\prime)$. With $\Psi,\Phi^T$ as above, we have
\begin{equation*}
\small
\rho_{U_r}(x^{-1})=
\begin{pmatrix}
	\Psi&0&0\\
	0&\Psi&0\\
	0&0&\Psi
\end{pmatrix}
\qquad\text{and}\qquad
\rho_{U_r}(y)=
\begin{pmatrix}
	0&0&\Phi^T\\
	\Phi^T&0&0\\
	0&\Phi^T&0
\end{pmatrix}
	.
\end{equation*}
\end{propn}

\begin{proof}
Start by labelling $g_i=(\nu^r-1)\nu^i\otimes\nu^{i-1}$, $g_{p+j}=(\nu^{\alpha(r+1)-1}-1)\nu^j\otimes\nu^{j-1}$ and $g_{2p+k}=(\nu^{\alpha^2(r+1)-1}-1)\nu^k\otimes\nu^{k-1}$, for $1\leq i,\,j,\,k\leq p$. It is clear that $\eta(U(r)^\prime+U(\alpha(r+1)-1)^\prime+U(\alpha^2(r+1)-1)^\prime)$ has $\mathbb{Z}$-basis $\{\eta(g_i),\,\eta(g_{p+j}),\,\eta(g_{2p+k})\:|\:1\leq i,\,j,\,k\leq p\}$.

The matrix $\rho_{U_r}(x^{-1})$ is clear. As such, we focus only on $\rho_{U_r}(y)$. We have,
\begin{eqnarray*}
g_i\cdot y^2=&(\nu^{\alpha r}-1)\nu^{\alpha i}\otimes \nu^{\alpha(i-1)}\\
=&(\nu^{\alpha(r+1)-1}-\nu^{\alpha-1})\nu^{\alpha(i-1)+1}\otimes \nu^{\alpha(i-1)}\\
=&(\nu^{\alpha(r+1)-1}-1)\nu^{\alpha(i-1)+1}\otimes\nu^{\alpha(i-1)}-(\nu^{\alpha-1}-1)\nu^{\alpha(i-1)+1}\otimes\nu^{\alpha(i-1)}.
\end{eqnarray*}
It follows that $\eta(g_i)y^2=\eta(g_{p+\alpha(i-1)+1})$. Likewise,
\[
g_{p+i}\cdot y^2=(\nu^{\alpha^2(r+1)-\alpha}-1)\nu^{\alpha i}\otimes \nu^{\alpha(i-1)}=(\nu^{\alpha^2(r+1)-1}-1)\nu^{\alpha(i-1)+1}\otimes\nu^{\alpha(i-1)}-(\nu^{\alpha-1}-1)\nu^{\alpha(i-1)+1}\otimes\nu^{\alpha(i-1)},
\]
and finally,
\[
g_{2p+i}\cdot y^2=(\nu^{r+1-\alpha}-1)\nu^{\alpha i}\otimes \nu^{\alpha(i-1)}=(\nu^{r}-1)\nu^{\alpha(i-1)+1}\otimes\nu^{\alpha(i-1)}-(\nu^{\alpha-1}-1)\nu^{\alpha(i-1)+1}\otimes\nu^{\alpha(i-1)}.
\]
It therefore follows that $\eta(g_{p+i})y^2=\eta(g_{2p+\alpha(i-1)+1})$ and $\eta(g_{2p+i})y^2=\eta(g_{\alpha(i-1)+1})$.
\end{proof}

We have therefore shown that $x,\,y$ act the same on both $U_r$ and $V_r$ (see Section \ref{sec:theorem_B_part_1}). As we have already shown $V_r\cong\Lambda$, the following is immediate:

\begin{propn}
\label{Free_part_after_quotienting_individual}
$\eta(U(r)^\prime+U(\alpha(r+1)-1)^\prime+U(\alpha^2(r+1)-1)^\prime)\cong\Lambda$ for $r\neq \alpha-1,\,\alpha^2-1$.
\end{propn}

\begin{corollary}
\label{Free_part_after_quotienting}
$(S+U^\prime)/(S+U_\alpha)\cong\Lambda^{d-1}$.
\end{corollary}

\begin{proof}
When defining $U_r$, we suppose $r\neq\alpha-1,\,\alpha^2-1$. In particular, we note that neither $\alpha(r+1)-1$ nor $\alpha^2(r+1)-1$ can equal $\alpha-1$ or $\alpha^2-1$. For suppose otherwise. If $\alpha(r+1)-1\equiv\alpha-1\:(mod\:p)$, then $p|\alpha r$ and so $p|\alpha$ or $p|r$, both of which are not true. Likewise, if $\alpha(r+1)-1=\alpha^2-1$ then $p|\alpha(r+1-\alpha)$. Since $p\not|\alpha$, this means $r=\alpha-1$, a possibility we have precluded. Similar arguments apply for $\alpha^2(r+1)-1=\alpha-1$ or $\alpha^2-1$.

Next, we consider $U_r$ and $U_s$ for some $r\neq s$. If $\alpha(r+1)-1= s$, then this is equivalent to $r=\alpha^2(s+1)-1$ and $\alpha^2(r+1)-1=\alpha(s+1)-1$. In other words, we have just rearranged the $U(r)^\prime$. Similar arguments apply for any other combination. A straightforward counting argument now shows we have $d-1$ copies of $\eta(U_r)$, each of which gives one copy of $\Lambda$.
\end{proof}

We therefore have the following split exact sequence $0\to S+U_\alpha\to S+U^\prime\to\Lambda^{d-1}\to 0$. Thus, $S+U^\prime\cong_\Lambda \bar{I}_C\otimes\bar{I}_C^\ast\cong_\Lambda (S+U_\alpha)\oplus\Lambda^{d-1}$. Proposition \ref{S_plus_U_alpha_is_K} now gives the following:

\begin{propn}
\label{R(1)_otimes_R(2)_is_K_oplus_free}
$\bar{I}_C\otimes\bar{I}_C^\ast\cong K\oplus\Lambda^{d-1}$.
\end{propn}

\section{$R(1)\otimes L\sim R(3)$}
\label{sec:theorem_B_part_4}

In this section we show $\bar{I}_C\otimes L\cong \bar{I_C^\ast}\oplus\Lambda^{2d}$, where $L$ is as defined in Section \ref{sec:theorem_B_part_1}. It is tempting to conjecture $L\sim K(2)$, but at present this is unknown.

\begin{propn}
\label{Theorem_F4_exact_sequence_1}
There exists the following exact sequence of $\Lambda$-modules,
\[0\to K\to\Lambda^{2d+1}\to \bar{I}_C\otimes L\to 0.\]
\end{propn}

\begin{proof}
Following the proof of Proposition \ref{K(1)_sim_L_dual} we have the exact sequence $0\to \bar{I_C^\ast}\to\Lambda\to L\to 0$. We now apply the exact functor $\bar{I}_C\otimes -$ to yield $0\to\bar{I}_C\otimes\bar{I_C^\ast}\to\Lambda^{3d}\to\bar{I}_C\otimes L\to 0$. By Proposition \ref{R(1)_otimes_R(2)_is_K_oplus_free}, $\bar{I}_C\otimes\bar{I_C^\ast}\cong K\oplus\Lambda^{d-1}$ resulting in $0\to K\oplus\Lambda^{d-1}\to\Lambda^{3d}\to\bar{I}_C\otimes L\to 0$. The result now follows from Johnson's `destabilization theorem'.
\end{proof}

\begin{propn}
\label{Theorem_B4_main_result}
$\bar{I}_C\otimes L\cong \bar{I_C^\ast}\oplus\Lambda^{2d}$.
\end{propn}

\begin{proof}
Recall that we have the exact sequence $0\to K\to\Lambda\to\bar{I_C^\ast}\to 0$. Using this and Proposition \ref{Theorem_F4_exact_sequence_1}, we apply the dual form of Schanuel's Lemma to obtain, $(\bar{I}_C\otimes L)\oplus\Lambda\cong\bar{I_C}^\ast\oplus\Lambda^{2d+1}$. The result now follows from Proposition \ref{Ri_are_straight}
 (i.e. $[\bar{I_C^\ast}]$ is straight).
\end{proof}

\section{$R(1)\otimes R(2)\sim K(2)$}
\label{sec:theorem_B_part_3}

Finally, we show $\bar{I}_C\otimes\overline{(x-1)I_C}\cong Y^\ast\oplus\Lambda^{d-1}$. First, observe the following $\mathbb{Z}$-basis for $\overline{(x-1)I_C}\otimes\bar{I_C^\ast}$: $\{(\nu-1)^2\nu^i\otimes\nu^j\:|\:0\leq i,\,j\leq p-2\}$. We work with the above $\mathbb{Z}$-basis and then dualise. We define the following modules for $1\leq r\leq p-2$, 
\begin{equation}
W(r)=span_\mathbb{Z}\{(\nu-1)^2\nu^{r+k}\otimes\nu^k\:|\:0\leq k\leq p-1\}\subset (x-1)I_C\otimes I_C^\ast.
\end{equation}
The arguments of Section \ref{sec:preliminaries_1} apply directly to the present modules $W(r)$, since they use only the action of $x$. Indeed, multiplication by $(\nu-1)^2$ induces a
$\Lambda_0$-isomorphism $\Phi:I_C^\ast\otimes I_C^\ast\to (x-1)I_C\otimes I_C^\ast$, where $a\otimes b\mapsto(\nu-1)^2a\otimes b$. If $V(r)$ denotes the module constructed in Section \ref{sec:preliminaries_1}, then $\Phi\bigl(\nu^{r+k}\otimes\nu^k\bigr)
=(\nu-1)^2\nu^{r+k}\otimes\nu^k$, and hence $\Phi(V(r))=W(r)$. Consequently, $W:=\bigoplus_{r=1}^{p-2}W(r)\cong_{\Lambda_0}\Lambda_0^{p-2}$. Moreover, if $T$ is as before, then $R=\Phi(T)=\sum^{p-2}_{i=0}(\sum^{p-2}_{j=i}(\nu-1)^2\nu^i\otimes\nu^j)$. Since $\Phi$ is a $\Lambda_0$-isomorphism and $I_C^\ast\otimes I^\ast_C =[T)\oplus V$, we obtain $(x-1)I_C\otimes I_C^\ast=[R)\oplus W\cong_{\Lambda_0}\mathbb{Z}\oplus\Lambda_0^{p-2}$. In summary, all relevant $\Lambda_0$-module conclusions carry over unchanged.

If we were to mirror the techniques of Section \ref{sec:theorem_B_part_1} directly, then we would encounter difficulties when introducing the $y$-action. Such difficulties arise due to the extra factor $(\nu-1)^2$ which becomes $(\nu^\alpha-1)^2$ when applying the cyclotomic $y$-action. This latter term is problematic using a direct adoption of the techniques of Section \ref{sec:theorem_B_part_1}. As such, our approach differs slightly, though the philosophy remains the same; namely, introduce the $y$-action and find the `free part'. What remains shall be the module $Y$.

We first set $E=span_{\mathbb{Z}}\{\nu,\nu^2,\ldots,\nu^{p-2}\}\subseteq I_C^\ast$. Thus, $I_C^\ast=\mathbb{Z}\oplus E$. For $f\in E$, define
\begin{equation}
W(f)=
span_{\mathbb Z}
\left\{
(\nu-1)^2f\nu^k\otimes\nu^k
\mid 0\leq k\leq p-1
\right\}.
\end{equation}
In particular, for $1\leq r\leq p-2$, $W(\nu^r)=W(r)$. Now, let $\sigma_\alpha:I_C^\ast\to I_C^\ast$ be the cyclotomic Galois automorphism given by $\sigma_\alpha(\nu)=\nu^\alpha$, and set $s_\alpha:=1+\nu+\cdots+\nu^{\alpha-1}$. Define an additive map $\tau:I_C^\ast\longrightarrow I_C^\ast$ by $\tau(f)=s_\alpha^2\sigma_\alpha(f)$.

\begin{lemma}
The map $\tau$ is a Galois structure of order three in which $W(f)y^{-1}=W(\tau(f))$ for every $f\in E$ such that $\tau(f)\in E$.
\end{lemma}

\begin{proof}
First, we show $\tau$ is a Galois structure of order three. It is clear that $\tau$ is additive since $\sigma_\alpha$ is. Moreover, since the action of $x$ on $\bar{I}_C^\ast$ is multiplication by \(\nu\), we have, $\tau(f\cdot x)=s_\alpha^2\sigma_\alpha(f\cdot x)=s_\alpha^2\sigma_\alpha(f)\nu^\alpha=\tau(f)x^\alpha=\tau(f)\theta(x)$. It therefore remains to show $\tau$ has order three. Note that $\nu^\alpha-1=(\nu-1)(1+\nu+\cdots+\nu^{\alpha-1})$ and so we have $(\nu^\alpha-1)^2=(\nu-1)^2s_\alpha^2$. Next, $\tau^2(f) =
\tau\bigl(s_\alpha^2\sigma_\alpha(f)\bigr)
= s_\alpha^2\sigma_\alpha(s_\alpha)^2
\sigma_\alpha^2(f)$, and hence $\tau^3(f)=s_\alpha^2\sigma_\alpha(s_\alpha)^2\sigma_\alpha^2(s_\alpha)^2\sigma_\alpha^3(f)$. Since $\sigma_\alpha^3=\operatorname{Id}$, $\sigma_\alpha^3(f)=f$, and so $\tau^3(f)=(s_\alpha\sigma_\alpha(s_\alpha)\sigma_\alpha^2(s_\alpha))^2f$. Using, $s_\alpha = (\nu^\alpha-1)/(\nu-1)$, we have,
\[
s_\alpha\sigma_\alpha(s_\alpha)\sigma_\alpha^2(s_\alpha)=
\frac{\nu^\alpha-1}{\nu-1}
\frac{\nu^{\alpha^2}-1}{\nu^\alpha-1}
\frac{\nu^{\alpha^3}-1}{\nu^{\alpha^2}-1}=1,
\]
where we have used $\alpha^3\equiv1\pmod p$. It follows that $\tau^3(f)=f$.

Finally, we suppose that $f\in E$ and $\tau(f)\in E$. By once more using $(\nu^\alpha-1)^2=(\nu-1)^2s_\alpha^2$,
\[
\left(
(\nu-1)^2f\nu^k\otimes\nu^k
\right)y^{-1}
=
(\nu^\alpha-1)^2 \sigma_\alpha(f)\nu^{\alpha k}\otimes\nu^{\alpha k} 
=
(\nu-1)^2\tau(f)\nu^{\alpha k} \otimes\nu^{\alpha k}.
\]
Since multiplication by $\alpha$ permutes the elements of
$\mathbb Z/p\mathbb Z$, this proves $W(f)y^{-1}=W(\tau(f))$. 
\end{proof}

\noindent For the Galois lattice $(I_C^\ast,\tau)$, the following is now clear.

\begin{propn}
$(I_C^\ast,\tau)\cong_\Lambda\overline{(x-1)I_C}$.
\end{propn}

Now let $\lambda:I_C^\ast\longrightarrow\mathbb{Z}$ be the $\mathbb{Z}$-module homomorphism which extracts the coefficient of $1$ with respect to the basis $\{1,\nu,\ldots,\nu^{p-2}\}$. This need not be a $\Lambda_0$-homomorphism nor even a ring homomorphism. However, observe $E=\ker(\lambda)$. Define
\[
\mathcal P
=
\left\{
f\in(I_C^\ast,\tau)
\:|\:
\lambda(f)=\lambda(\tau(f))
=\lambda(\tau^2(f))=0
\right\}.
\]
Equivalently, $\mathcal P=E\cap\tau(E)\cap\tau^2(E)$. In particular, $\mathcal P$ is invariant under $\tau$.

Restricting scalars along $j:\mathbb{Z}[C_3]\hookrightarrow\Lambda$ gives
\[
j^\ast(I_C^\ast,\tau)
\cong_{\mathbb Z[C_3]}
j^\ast(\overline{(x-1)I_C}).
\]
Define
\[
\delta:j^\ast(I_C^\ast,\tau)\longrightarrow\mathbb{Z}[C_3]
\quad \text{ by }\quad
\delta(f)
=
\lambda(f)
+\lambda(\tau(f))y
+\lambda(\tau^2(f))y^2.
\]

\begin{lemma}
The map $\delta$ is a homomorphism of right
$\mathbb Z[C_3]$-modules and $\ker(\delta)=\mathcal P$. 
\end{lemma}

\begin{proof}
The action of $y^{-1}$ on $j^\ast(I_C^\ast,\tau)$ is given by $f\cdot y^{-1}=\tau(f)$. Consequently,
\[
\delta(fy^{-1})=
\delta(\tau(f))=
\lambda(\tau(f))
+\lambda(\tau^2(f))y
+\lambda(f)y^2 =
\delta(f)y^{-1}.
\]
Thus, $\delta$ is a $\mathbb{Z}[C_3]$-homomorphism. The assertion
concerning its kernel follows immediately from definition.
\end{proof}

\noindent The following is the principal observation:

\begin{propn}
The homomorphism $\delta:j^\ast(I_C^\ast,\,\tau)\longrightarrow\mathbb Z[C_3]$ is surjective.
\end{propn}

\begin{proof}
Put $\mathcal O=\mathbb Z[\nu]$, $F=\mathbb Q(\nu)$, $J=(\nu-1)^2\mathcal O$. Let $\mu:j^\ast(I_C^\ast,\tau)\longrightarrow J$ be the isomorphism of $\mathbb Z[C_3]$-lattices given by $\mu(f)=(\nu-1)^2f$. Transporting $\lambda$ through $\mu$, define $\ell=\lambda\circ\mu^{-1}:J\longrightarrow\mathbb Z$. Moreover, $\mu\circ\tau=\sigma_\alpha\circ\mu$. Indeed,
\[
\mu(\tau(f))
=(\nu-1)^2s_\alpha^2\sigma_\alpha(f)
=(\nu^\alpha-1)^2\sigma_\alpha(f)
=\sigma_\alpha\bigl((\nu-1)^2f\bigr)
=\sigma_\alpha(\mu(f)).
\]

The trace-dual of $\mathcal O$ is
\[
\mathcal O^\#
=
\frac{1-\nu}{p}\mathcal O.
\]
For any $a\in F^\times$, we have $(a\mathcal O)^\#=a^{-1}\mathcal O^\#$. Since $J=(\nu-1)^2\mathcal O$, it follows that
\[
J^\#
=(\nu-1)^{-2}\mathcal O^\#
=(\nu-1)^{-2}\frac{1-\nu}{p}\mathcal O
=\frac{1}{p(1-\nu)}\mathcal O.
\]

Moreover, $(1-\nu)/p$ is the element of $\mathcal O^\#$ representing
the coefficient functional $\lambda$. Indeed,
\[
\operatorname{Tr}_{F/\mathbb Q}
\left(\frac{1-\nu}{p}\nu^j\right)
=
\begin{cases}
1,&j=0,\\
0,&1\leq j\leq p-2,
\end{cases}
\]
since $\operatorname{Tr}_{F/\mathbb Q}(1)=p-1$ and $\operatorname{Tr}_{F/\mathbb Q}(\nu^m)=-1$ whenever $p\nmid m$. Consequently, if $f=\sum_{j=0}^{p-2}a_j\nu^j$, then
\[
\operatorname{Tr}_{F/\mathbb Q}
\left(\frac{1-\nu}{p}f\right)
=a_0=\lambda(f).
\]

It follows that, under the trace-duality identification $J^\#\cong\operatorname{Hom}_{\mathbb Z}(J,\mathbb Z)$, the functional $\ell$ is represented by
\[
\xi=\frac{1}{p(1-\nu)}.
\]
Indeed, if $z=\mu(f)=(\nu-1)^2f$, then
\[
\operatorname{Tr}_{F/\mathbb Q}(\xi z)
=
\operatorname{Tr}_{F/\mathbb Q}
\left(
\frac{(\nu-1)^2}{p(1-\nu)}f
\right)
=
\operatorname{Tr}_{F/\mathbb Q}
\left(\frac{1-\nu}{p}f\right)
=\lambda(f)
=\ell(z).
\]

Let $\varepsilon_0,\varepsilon_1,\varepsilon_2$ be the dual basis of $\operatorname{Hom}_{\mathbb Z}(\mathbb Z[C_3],\mathbb Z)$ corresponding to the basis $1,y,y^2$. By the definition of $\delta$,
\[
\delta^\ast(\varepsilon_0)=\lambda,
\qquad
\delta^\ast(\varepsilon_1)=\lambda\circ\tau,
\qquad
\delta^\ast(\varepsilon_2)=\lambda\circ\tau^2.
\]
After transporting these functionals through $\mu$, and using
$\mu\circ\tau=\sigma_\alpha\circ\mu$, they become, respectively,
\[
\ell,
\qquad
\ell\circ\sigma_\alpha,
\qquad
\ell\circ\sigma_{\alpha^2}.
\]

Let $\beta$ be the least positive residue of $\alpha^2$ modulo $p$.
Since $\alpha$ has order three, $\alpha^{-1}\equiv\beta\pmod p$. For $z\in J$, invariance of the trace gives
\[
\ell(\sigma_\alpha z)
=
\operatorname{Tr}_{F/\mathbb Q}
\bigl(\xi\sigma_\alpha(z)\bigr)\\
=
\operatorname{Tr}_{F/\mathbb Q}
\bigl(\sigma_\beta(\xi)z\bigr).
\]
Similarly, $\ell\circ\sigma_{\alpha^2}$ is represented by
$\sigma_\alpha(\xi)$. Hence, under the isomorphism $\mathcal O\longrightarrow J^\#$, given by $a\longmapsto\xi a$, the image of $\delta^\ast$ corresponds to the sublattice of
$\mathcal O$ generated by
\[
1,
\qquad
u=\frac{1-\nu}{1-\nu^\alpha},
\qquad
v=\frac{1-\nu}{1-\nu^\beta}.
\]

Since $\alpha\beta\equiv1\pmod p$, we have
\[
u=\sum_{t=0}^{\beta-1}\nu^{\alpha t},
\qquad
v=\sum_{t=0}^{\alpha-1}\nu^{\beta t}.
\]
Furthermore, $\alpha+\beta+1\equiv0\pmod p$. Since $\alpha$ and $\beta$ lie in $\{1,\ldots,p-1\}$, this implies $\alpha+\beta=p-1$. 

We claim that $1,u,v$ generate a primitive rank-three sublattice of $\mathcal O$. Let $\ell_0$ be a rational prime and suppose that $c_0+c_1u+c_2v=0$ in $\mathcal O/\ell_0\mathcal O$, where $c_0,c_1,c_2\in\mathbb F_{\ell_0}$. Set
\[
\mathcal U
=
\{\alpha t:0\leq t\leq\beta-1\},
\qquad
\mathcal V
=
\{\beta t:0\leq t\leq\alpha-1\},
\]
regarded as subsets of $\mathbb Z/p\mathbb Z$.

Under the identification $\mathcal O/\ell_0\mathcal O\cong\mathbb F_{\ell_0}[C_p]/(\Sigma_x)$, the coefficient function of $c_0+c_1u+c_2v$ must be constant. Now, $|\mathcal U|=\beta$, $|\mathcal V|=\alpha$, and both sets contain $0$. Since $\alpha+\beta=p-1$, it follows that
\[
\left|
\mathbb Z/p\mathbb Z\setminus
(\mathcal U\cup\mathcal V)
\right|
=
p-|\mathcal U|-|\mathcal V|
+|\mathcal U\cap\mathcal V|
=
1+|\mathcal U\cap\mathcal V|
\geq2.
\]
Evaluating the coefficient function at an element outside
$\mathcal U\cup\mathcal V$ shows that this constant is zero.

Since $\alpha\neq\beta$ (otherwise
$\alpha^2\equiv\alpha\pmod p$, contradicting the fact that $\alpha$
has order three), the sets $\mathcal U$ and $\mathcal V$ have
different cardinalities. Suppose first that
$|\mathcal U|>|\mathcal V|$; the other case is symmetric. Then $\mathcal U\setminus\mathcal V\neq\varnothing$. An element of this difference is necessarily nonzero, because
$0\in\mathcal U\cap\mathcal V$. Evaluating the coefficient function there gives $c_1=0$. 

Since $|\mathcal V|=\alpha\geq2$, the set $\mathcal V$ contains a
nonzero element. Evaluating there, and using $c_1=0$, gives $c_2=0$. Finally, evaluating at $0$ gives $c_0=0$. The case $|\mathcal V|>|\mathcal U|$ follows by interchanging
$(\mathcal U,u,c_1)$ and $(\mathcal V,v,c_2)$.

Thus, $1,u,v$ are linearly independent modulo every rational prime.
They therefore generate a primitive rank-three sublattice of
$\mathcal O$. Multiplication by $\xi$ identifies $\mathcal O$ with
$J^\#$, so $\operatorname{Im}(\delta^\ast)$ is a primitive
rank-three sublattice of $\operatorname{Hom}_{\mathbb Z}
\bigl(j^\ast(I_C^\ast,\tau),\mathbb Z\bigr)$. 

The integer matrices of $\delta$ and $\delta^\ast$ are transposes of one another and hence have the same nonzero Smith invariants. Since the image of $\delta^\ast$ is primitive, all three nonzero Smith invariants are equal to one. Therefore $\operatorname{coker}(\delta)=0$, and $\delta$ is surjective.
\end{proof}

\begin{propn}
There is an isomorphism of $\mathbb Z[C_3]$-modules, $\mathcal P\cong\mathbb Z[C_3]^{d-1}$.
\end{propn}

\begin{proof}
Recall, $j^\ast(I_C^\ast,\tau)
\stackrel{\simeq}{\longrightarrow}
j^*(\overline{(x-1)I_C})$. Since $\overline{(x-1)I_C}\cong R(2)$, Proposition \ref{Ri_as_ZC3_mods} gives $j^\ast(I_C^\ast,\tau) \cong \mathbb{Z}[C_3]^d$. We therefore have a short exact sequence
\[
0\longrightarrow\mathcal P
\longrightarrow\mathbb Z[C_3]^d
\overset{\delta}{\longrightarrow}
\mathbb Z[C_3]
\longrightarrow0.
\]
This clearly splits, meaning $\mathcal{P}$ is stably free of rank $d-1$ and hence free, as $\mathbb{Z}[C_3]$ has stably free cancellation, i.e. $\mathcal P\cong\mathbb Z[C_3]^{d-1}$.
\end{proof}

Choose a $\mathbb Z[C_3]$-basis, $\{f_1,\ldots,f_{d-1}\}$, of $\mathcal P$. Then
\[
\mathcal B_{\mathcal P}
=
\left\{
f_i,\tau(f_i),\tau^2(f_i)
\mid
1\leq i\leq d-1
\right\}
\]
is a $\mathbb Z$-basis of $\mathcal P$. For each $1\leq i\leq d-1$, define
\[
W_i:=W(f_i)+W(\tau(f_i))+W(\tau^2(f_i)).
\]

\begin{propn}
For each $1\leq i\leq d-1$, $W_i\cong_\Lambda\Lambda$.
\end{propn}

\begin{proof}
For $1\leq j\leq p$, set $e_j=
(\nu-1)^2f_i\nu^{j-1}\otimes\nu^{j-1}$, $
e_{p+j}=(\nu-1)^2\tau(f_i)\nu^{j-1}\otimes\nu^{j-1}$ and $e_{2p+j}=(\nu-1)^2\tau^2(f_i)\nu^{j-1}\otimes\nu^{j-1}$. We first show these elements form a $\mathbb Z$-basis of $W_i$, for which we need only show linear independence. Recall,
\[\{
u_{r,j}
=
(\nu-1)^2\nu^{r+j-1}\otimes\nu^{j-1}\:|\:
1\leq r\leq p-2,\quad 1\leq j\leq p\}
\]
forms a $\mathbb{Z}$-basis of $W$. Since
$f_i,\,\tau(f_i),\,\tau^2(f_i)\in E$, write $f_i=\sum_{r\geq 1}a_r\nu^r,\,\tau(f_i)=\sum_{r\geq 1}b_r\nu^r$ and $\tau^2(f_i)=\sum_{r\geq 1}c_r\nu^r$. It follows that $e_j=\sum_ra_ru_{r,j},\,e_{p+j}=\sum_rb_ru_{r,j},\,e_{2p+j}=\sum_rc_ru_{r,j}.$ Next, suppose $\sum_{j=1}^p
\left(A_je_j+B_je_{p+j}+C_je_{2p+j}\right)=0$. The linear independence of the $u_{r,j}$ implies, for each $j$, $A_jf_i+B_j\tau(f_i)+C_j\tau^2(f_i)=0$. Since the elements $f_i,\,\tau(f_i),\,\tau^2(f_i)$ are $\mathbb{Z}$-linearly independent, $A_j=B_j=C_j=0$ for every $j$.

Now we consider $x$ and $y$-actions. The $x$-action is given by $e_jx=e_{j+1}$, $e_{p+j}x=e_{p+j+1}$ and $e_{2p+j}x=e_{2p+j+1}$, where the indices in each block are read mod $p$. This gives rise to the following:
\[
\rho_{W_i}(x^{-1})
=
\begin{pmatrix}
\Psi&0&0\\
0&\Psi&0\\
0&0&\Psi
\end{pmatrix}.
\]

For the $y$-action, we stress that the left hand side of the tensor product is the module $\overline{(x-1)I_C}$, and so $y$ acts by the cyclotomic action. As such, $e_jy^{-1}=(\nu^\alpha -1)^2\sigma_\alpha(f_i)\nu^{\alpha(j-1)}\otimes\nu^{\alpha(j-1)}$. Since $(\nu^\alpha-1)^2=(\nu-1)^2s_\alpha^2$, the left hand side of the tensor product becomes $(\nu-1)^2\tau(f_i)\nu^{\alpha(j-1)}$ by the definition of $\tau$. Consequently, $e_jy^{-1}=e_{p+\alpha(j-1)+1}$. Likewise, $e_{p+j}y^{-1}=e_{2p+\alpha(j-1)+1}$, and $e_{2p+j}y^{-1}=e_{\alpha(j-1)+1}$, where the subscripts are again taken mod $p$. Therefore,
\[
\rho_{W_i}(y)
=
\begin{pmatrix}
0&0&\Phi^T\\
\Phi^T&0&0\\
0&\Phi^T&0
\end{pmatrix}.
\]
These are precisely the matrices obtained for the $V_r$ in Section \ref{sec:theorem_B_part_1}. Recall, we showed $V_r$ was free. It follows that $W_i\cong_\Lambda\Lambda$.
\end{proof}

\begin{propn}
    The sum $W_1+\cdots+ W_{d-1}$ is direct.
\end{propn}

\begin{proof}
    For $f\in E$ and $0\leq k\leq p-1$, write $w_k(f) =
(\nu-1)^2f\nu^k\otimes\nu^k$. For each $k$, define the $\mathbb{Z}$-linear map
\[
\theta_k:E\longrightarrow W,
\qquad
\theta_k(f)=w_k(f).
\]
Indeed, every $f\in E$ has the unique form $f=\sum^{p-2}_{r=1}a_r\nu^r$. Consequently, $\theta_k(f)=\sum^{p-2}_{r=1}a_r(\nu-1)^2\nu^{k+r}\otimes \nu^k\in W=\oplus_{r=1}^{p-2} W(r)$.

Recall that $\mathcal B_W =\left\{
w_k(\nu^r)
\mid
1\leq r\leq p-2,\ 0\leq k\leq p-1
\right\}$ is a $\mathbb{Z}$-basis of $W$. If we fix $0\leq k\leq p-1$, then since $E = span_{\mathbb{Z}}\{\nu,\ldots,\nu^{p-2}\}$, we have $\theta_k(E)=span_{\mathbb{Z}}\left\{w_k(\nu^r)\:|\:1\leq r\leq p-2\right\}$. In particular, the elements of this spanning set form a subset of the basis $\mathcal{B}_W$ and are therefore $\mathbb{Z}$-linearly independent. As such, for each $0\leq k\leq p-1$, we can set $\mathcal{B}_k=\left\{w_k(\nu^r)\:|\:1\leq r\leq p-2\right\}$, which is a $\mathbb{Z}$-basis of $\theta_k(E)$. Moreover,
\[
\mathcal{B}_W
=
\bigcup_{k=0}^{p-1}\mathcal{B}_k.
\]
Since $\mathcal{B}_W$ is a $\mathbb{Z}$-basis of $W$, it follows that
\[
W
=
span_{\mathbb{Z}}(\mathcal{B}_W)
=
span_{\mathbb{Z}}
\left(
\bigcup_{k=0}^{p-1}\mathcal{B}_k
\right)
=
\sum_{k=0}^{p-1}
span_{\mathbb{Z}}(\mathcal{B}_k)
=
\sum_{k=0}^{p-1}\theta_k(E).
\]
This sum is necessarily direct. For suppose that $\sum_{k=0}^{p-1}z_k=0$, where $z_k\in\theta_k(E)$. For each $k$, we write $z_k=\sum_{r=1}^{p-2}a_{r,k}w_k(\nu^r)$. Then $0=\sum_{k=0}^{p-1}\sum_{r=1}^{p-2}a_{r,k}w_k(\nu^r)$. Since $\mathcal B_W$ is a $\mathbb Z$-basis of $W$, all the
coefficients $a_{r,k}$ vanish. Hence $z_k=0$ for every $k$. Therefore,
\[
W
=
\bigoplus_{k=0}^{p-1}\theta_k(E).
\]
Take the $\mathbb{Z}$-basis for $\mathcal{P}$, $\mathcal{B}_{\mathcal{P}} = \left\{f_i,\tau(f_i),\tau^2(f_i)\:|\:1\leq i\leq d-1\right\}$. For each $1\leq i\leq d-1$, set
\[
\mathcal P_i
:=
\mathbb{Z} f_i
\oplus
\mathbb{Z}\tau(f_i)
\oplus
\mathbb{Z}\tau^2(f_i)
= f_i\mathbb{Z}[C_3].
\]
Then $\mathcal{P}=\bigoplus_{i=1}^{d-1}\mathcal{P}_i$. By definition, $W_i=W(f_i)+W(\tau(f_i))+W(\tau^2(f_i))$. For any $f\in E$, we have $W(f)=\sum_{k=0}^{p-1}\mathbb{Z}\theta_k(f)$. 
Therefore,
\begin{eqnarray*}
W_i
=&
\sum_{k=0}^{p-1}\mathbb Z\theta_k(f_i)
+
\sum_{k=0}^{p-1}\mathbb Z\theta_k(\tau(f_i))
+
\sum_{k=0}^{p-1}\mathbb Z\theta_k(\tau^2(f_i))\\
=&
\sum_{k=0}^{p-1}
\left(
\mathbb Z\theta_k(f_i)
+
\mathbb Z\theta_k(\tau(f_i))
+
\mathbb Z\theta_k(\tau^2(f_i))
\right).
\end{eqnarray*}
Using the definition of $\mathcal P_i$, along with the fact $\theta_k$ is $\mathbb Z$-linear, we have
\begin{eqnarray*}
\theta_k(\mathcal{P}_i)
=&
\theta_k\left(
\mathbb Zf_i+
\mathbb Z\tau(f_i)+
\mathbb Z\tau^2(f_i)
\right)\\
=&
\mathbb Z\theta_k(f_i)
+
\mathbb Z\theta_k(\tau(f_i))
+
\mathbb Z\theta_k(\tau^2(f_i)).
\end{eqnarray*}
Hence $W_i=\sum_{k=0}^{p-1}\theta_k(\mathcal{P}_i)$. Moreover, since $\theta_k(\mathcal{P}_i)\subseteq\theta_k(E)$ and $W=\bigoplus_{k=0}^{p-1}\theta_k(E)$, this sum is direct. Thus
\[
W_i
=
\bigoplus_{k=0}^{p-1}\theta_k(\mathcal P_i).
\]

Finally, we prove that the sum of $W_i$ is direct. Suppose that $z_1+\cdots+z_{d-1}=0$, where $z_i\in W_i$. For each $i$, write uniquely $z_i=\sum_{k=0}^{p-1}\theta_k(g_{i,k})$, where $g_{i,k}\in\mathcal{P}_i$. Then
\[
0
=
\sum_{i=1}^{d-1}z_i
=
\sum_{k=0}^{p-1}
\theta_k\left(\sum_{i=1}^{d-1}g_{i,k}\right).
\]
Since $W=\bigoplus_{k=0}^{p-1}\theta_k(E)$, we obtain $\theta_k\left(\sum_{i=1}^{d-1}g_{i,k}\right)=0$ for every $k$. However, $\theta_k$ is injective, and so $\sum_{i=1}^{d-1}g_{i,k}=0$ for every $k$. Now, $\mathcal P=\bigoplus_{i=1}^{d-1}\mathcal P_i$, and each $g_{i,k}$ lies in $\mathcal P_i$. Hence $g_{i,k}=0$ for every $i$ and $k$. It follows that $z_i=0$ for every $i$. The result follows.
\end{proof}

We may therefore set
\[
W_f
:=
W_1\oplus\cdots\oplus W_{d-1}.
\]
Since each $W_i$ is isomorphic to the regular $\Lambda$-lattice,
we conclude that $W_f\cong_\Lambda \Lambda^{d-1}$. It remains to show $\overline{(x-1) I_C}\otimes\bar I_C^\ast/W_f$ is torsion free.

\begin{propn}
The defining $\mathbb Z$-basis of $W_f$ is contained in a
$\mathbb Z$-basis of
$\overline{(x-1)I_C}\otimes \bar{I}_C^\ast$.
\end{propn}

\begin{proof}
We show this by first showing the defining $\mathbb Z$-basis of $W_f$ is contained in a
$\mathbb Z$-basis of $W$. Consider the homomorphism of abelian groups, $\chi:E\longrightarrow\mathbb{Z}^2$ defined by $\chi(f)=
\bigl(
\lambda(\tau(f)),
\lambda(\tau^2(f))
\bigr)$. Since $E=\ker(\lambda)$, we have $\ker(\chi)=\mathcal{P}$. It follows that $E/\mathcal{P}\cong Im(\chi)\subseteq\mathbb{Z}^2$. As every subgroup of a free abelian group is free abelian, $E/\mathcal P$ is torsion free.

Since $rank_{\mathbb{Z}}(E)=p-2$ and $rank_{\mathbb{Z}}(\mathcal{P})=p-4$, we may choose $g_1,g_2\in E$ such that
\[
\mathcal B_E
=
\mathcal{B}_{\mathcal{P}}\cup\{g_1,g_2\}
\]
is a $\mathbb{Z}$-basis of $E$, where $\mathcal{B}_{\mathcal{P}}
=
\left\{
\tau^j(f_i)
\;\middle|\;
1\leq i\leq d-1,\ 0\leq j\leq2
\right\}$.

For each $0\leq k\leq p-1$, the map $\theta_k:E\longrightarrow\theta_k(E)$ is an isomorphism of abelian groups. Consequently, $\theta_k(\mathcal{B}_E)$ is a $\mathbb{Z}$-basis of $\theta_k(E)$. Since $\theta_k(h)=(\nu-1)^2h\nu^k\otimes\nu^k$ for every $h\in E$, this basis is given explicitly by
\begin{eqnarray*}
\theta_k(\mathcal{B}_E)
={}&
\left\{
(\nu-1)^2\tau^j(f_i)\nu^k\otimes\nu^k
\;\middle|\;
1\leq i\leq d-1,\ 0\leq j\leq2
\right\}\\
&\cup
\left\{
(\nu-1)^2g_s\nu^k\otimes\nu^k
\;\middle|\;
s=1,2
\right\}.
\end{eqnarray*}
Moreover, $W=\bigoplus_{k=0}^{p-1}\theta_k(E)$. It follows that the union of these bases, as $k$ ranges from $0$
to $p-1$, is a $\mathbb Z$-basis of $W$. Thus, the following is a $\mathbb{Z}$-basis of $W$:
\begin{eqnarray*}
\mathcal{B}_W'
={}&
\left\{
(\nu-1)^2\tau^j(f_i)\nu^k\otimes\nu^k
\;\middle|\;
\begin{array}{c}
1\leq i\leq d-1,\quad 0\leq j\leq2,\\
0\leq k\leq p-1
\end{array}
\right\}\\
&\cup
\left\{
(\nu-1)^2g_s\nu^k\otimes\nu^k
\;\middle|\;
s=1,2,\quad 0\leq k\leq p-1
\right\}.
\end{eqnarray*}
The first set is precisely the defining basis of $W_f$. Thus, the basis of $W_f$ extends to a basis of $W$. By the $\Lambda_0$-decomposition obtained above,
\[
\overline{(x-1)I_C}\otimes\bar{I_C^\ast}
\cong_{\Lambda_0}
[R)\oplus W,
\]
the basis of $W$ itself extends to a $\mathbb{Z}$-basis of
$\overline{(x-1) I_C}\otimes\bar{I_C^\ast}$. The result follows.
\end{proof}

Let $\kappa:\overline{(x-1) I_C}\otimes\bar{I_C^\ast}
\longrightarrow
\left(
\overline{(x-1)I_C}\otimes\bar{I_C^\ast}
\right)/W_f$ be the natural surjection and set
\[
Y=
\left(
\overline{(x-1)I_C}\otimes\bar I_C^*
\right)/W_f.
\]
By the previous proposition, $Y$ is a lattice. Moreover, $rank_{\mathbb Z}(Y)= (p-1)^2-3p(d-1)=2p+1$.

\begin{propn}
\label{Iast_otimes_x-1I_iso_remainder_plus_free}
$\overline{(x-1)I_C}\otimes\bar{I_C^\ast}\cong_\Lambda Y\oplus\Lambda^{d-1}.$
\end{propn}

\begin{proof}
We have a short exact sequence of $\Lambda$-lattices $0\to W_f
\to\overline{(x-1)I_C}\otimes\bar{I}_C^\ast\overset{\kappa}{\longrightarrow} Y\to 0$. Dualising gives a split short exact sequence since $W_f^*\cong_\Lambda\Lambda^{d-1}$, and the result follows from dualising again.
\end{proof}

Finally, we show $Y^\ast\sim K(2)$. We would like to know if $Y\sim L^\ast$. However, at present this seems too great a computational task. Nevertheless, we will discuss to what extent this can be proven to be true in Section \ref{sec:theorem_B_part_2}.

\begin{propn}
\label{Yast_sim_K2}
$Y^\ast\sim K(2)$.
\end{propn}

\begin{proof}
Recall the decomposition, $I_G\cong\overline{I_C}\oplus[y-1)$ (Theorem 71.5 of \cite{Johnson2021}, p.230). It therefore follows that we have the exact sequence, $0\to \overline{I_C}\oplus[y-1)\to\Lambda\to\mathbb{Z}\to 0$. Next, we apply $-\otimes\overline{(x-1)I_C}$ and use Frobenius Reciprocity to obtain the exact sequence, $0 \to \overline{I_C}\otimes\overline{(x-1)I_C}\oplus\Lambda^{2d}\to\Lambda^{3d}\to\overline{(x-1)I_C}\to 0$. Using the dual form of Proposition \ref{Iast_otimes_x-1I_iso_remainder_plus_free} alongside Johnson's `destabilization theorem', we form the following exact sequence, $0 \to Y^\ast \to\Lambda\to\overline{(x-1)I_C}\to 0$. However, by definition, we also have $0 \to K(2) \to\Lambda\to\overline{(x-1)I_C}\to 0$. It therefore follows from Schanuel's Lemma that $K(2)\oplus\Lambda\cong Y^\ast\oplus\Lambda$.
\end{proof}

\begin{corollary}
$\overline{(x-1)I_C}\otimes\overline{I_C}\cong K(2)\oplus\Lambda^{d-1}$.
\end{corollary}

\section{$R(1)\otimes Y\sim R(2)$ and the $\mathcal{D}(2)$-problem}
\label{sec:theorem_B_part_2}

\begin{propn}
\label{R(1)_otimes_Y_iso_R(2)}
$\bar{I}_C\otimes Y\cong \overline{(x-1)I_C}\oplus\Lambda^{2d}$.
\end{propn}

\begin{proof}
From Proposition \ref{Iast_otimes_x-1I_iso_remainder_plus_free}, we have $\bar{I}^\ast_C\otimes\overline{(x-1)I_C}\cong Y\oplus\Lambda^{d-1}$. By applying $-\otimes\bar{I}_C$ and using Proposition \ref{R(1)_otimes_R(2)_is_K_oplus_free} we see that the left hand side is isomorphic to $(K\otimes\overline{(x-1)I_C})\oplus\Lambda^{3d(d-1)}$. Using Proposition \ref{Ri_otimes_K__iso_Ri}, this becomes $\overline{(x-1)I_C}\oplus\Lambda^{d(3d-1)}$. Next, the right hand side is isomorphic to $(\bar{I}_C\otimes Y)\oplus\Lambda^{3d(d-1)}$. From Proposition \ref{Ri_are_straight}, we know $\overline{(x-1)I_C}$ is straight and so we can cancel the free modules above. This leaves us with $\bar{I}_C\otimes Y\cong \overline{(x-1)I_C}\oplus\Lambda^{2d}$.
\end{proof}

At this point, it is clear that the issue lies in $L$ and $Y$. Should we be able to show a relationship between these two modules (namely, one is stably isomorphic to the dual of the other), then we would be able to fully realise the interactions of the syzygy modules under tensor product. In particular, this would lead to a solution of the $\mathcal{D}(2)$-problem in the case of $G(p,\,3)$. Assuming for now that such a relationship exists, we can prove the following which is Theorem C of the Introduction.

\begin{propn}
\label{L_Y_then_D2}
If $L\sim Y^\ast$, then $G(p,\,3)$ has the $\mathcal{D}(2)$-property.
\end{propn}

\begin{proof}
First, note that each $[K(i)]$ has the structure of a fork (see \cite{Johnson2003}). It follows that $Y^\ast\oplus\Lambda\cong L\oplus\Lambda$, and since $d\geq 2$, we can always replace $Y^\ast\oplus\Lambda^{d-1}$ by $L\oplus\Lambda^{d-1}$ in our exact sequences. Recall, in the proof of Proposition \ref{K(1)_sim_L_dual} we showed the existence of,
\begin{equation}
\label{D2_exact_seq_1}
0\to L^\ast\to \Lambda \to \bar{I_C} \to 0.
\end{equation}
We now apply $-\otimes\bar{I_C}$. Using Propositions \ref{IC_otimes_IC_iso_L_plus_free} and \ref{R(1)_otimes_Y_iso_R(2)}, Johnson's `destabilization theorem', and the assumption that $Y\sim L^\ast$, we end up with the following exact sequence:
\begin{equation}
\label{D2_exact_seq_2}
0 \to \overline{(x-1)I_C} \to \Lambda \to L^\ast \to 0.
\end{equation}
Splicing together (\ref{D2_exact_seq_1}) and (\ref{D2_exact_seq_2}) results in the following:
\begin{equation}
\label{Basic_seq_1}
\begin{tikzpicture}
\node at (-11.75,0){$0$};
\draw[thin,->](-11.5,0)--(-11,0)node[anchor=west]{$\overline{(x-1)I_C}$};
\draw[thin,->](-9,0)--(-8.5,0)node[anchor=west]{$\Lambda$};
\draw[thin,->](-8.1,0)--(-7.6,0)node[anchor=west]{$\Lambda$};
\draw[thin,->](-7,0)--(-6.5,0)node[anchor=west]{$\bar{I}_C$};
\draw[thin,->](-5.75,0)--(-5.25,0)node[anchor=west]{$0$.};

\node at (-7.75,1.1){$L^\ast$};
\draw[thin,->](-8.1,.3)--(-7.9,.8);
\draw[thin,->](-7.7,.8)--(-7.5,.3);
\end{tikzpicture}
\end{equation}
It therefore follows that $\overline{(x-1)I_C}\in\Omega_2(\bar{I}_C)$. From the exact sequence $0\to[y-1)\to\Lambda\to\Lambda\to[y-1)\to 0$, it follows that $[y-1)\in\Omega_2([y-1))$. Thus, $\overline{(x-1)I_C}\oplus[y-1)\in\Omega_2(\bar{I}_C\oplus[y-1))=\Omega_3(\mathbb{Z})$. The result now follows from (105.6) of \cite{Johnson2021}.
\end{proof}

\noindent We can in fact extend this argument to construct a free resolution of $\bar{I_C}$ with period 6.

\begin{propn}
If $L\sim Y^\ast$, then there exists a free resolution of $\bar{I}_C$ of period 6:
\begin{equation*}
\begin{tikzpicture}
\node at (-12.45,0){$0$};
\draw[thin,->](-12.2,0)--(-11.7,0)node[anchor=west]{$\bar{I}_C$};
\draw[thin,->](-11,0)--(-10.5,0)node[anchor=west]{$\Lambda$};
\draw[thin,->](-10,0)--(-9.5,0)node[anchor=west]{$\Lambda$};
\draw[thin,->](-9,0)--(-8.5,0)node[anchor=west]{$\Lambda$};
\draw[thin,->](-8,0)--(-7.5,0)node[anchor=west]{$\Lambda$};
\draw[thin,->](-7,0)--(-6.5,0)node[anchor=west]{$\Lambda$};
\draw[thin,->](-6,0)--(-5.5,0)node[anchor=west]{$\Lambda$};
\draw[thin,->](-5,0)--(-4.5,0)node[anchor=west]{$\bar{I}_C$};
\draw[thin,->](-3.8,0)--(-3.3,0)node[anchor=west]{$0$};

\node at (-5.7,1.1){$L^\ast$};
\draw[thin,->](-6.1,.3)--(-5.8,.8);
\draw[thin,->](-5.65,.8)--(-5.35,.3);

\node at (-7.7,1.1){$L$};
\draw[thin,->](-8.1,.3)--(-7.8,.8);
\draw[thin,->](-7.65,.8)--(-7.35,.3);

\node at (-9.7,1.1){$K$};
\draw[thin,->](-10.1,.3)--(-9.8,.8);
\draw[thin,->](-9.65,.8)--(-9.35,.3);

\node at (-6.7,-1.2){$\overline{(x-1)I_C}$};
\draw[thin,->](-7.1,-.3)--(-6.8,-.8);
\draw[thin,->](-6.65,-.8)--(-6.35,-.3);

\node at (-8.7,-1.2){$\bar{I^\ast_C}$};
\draw[thin,->](-9.1,-.3)--(-8.8,-.8);
\draw[thin,->](-8.65,-.8)--(-8.35,-.3);
\end{tikzpicture}
\end{equation*}
\end{propn}

\begin{proof}
In the above proof we constructed the first section of the free resolution. We now have two more sections to construct. We start with $0\to L \to \Lambda \to \overline{(x-1)I_C} \to 0$ which is the dual of (\ref{D2_exact_seq_2}). We apply $-\otimes \bar{I_C}$ and Johnson's `destabilization theorem' to obtain the exact sequence, $0 \to \bar{I_C^\ast} \to \Lambda \to L \to 0$. Here, we have used Proposition \ref{Theorem_B4_main_result}, the dual of Proposition \ref{Iast_otimes_x-1I_iso_remainder_plus_free}, along with the assumption that $Y^\ast\sim L$. Splicing these two exact sequences together yields the following:
\begin{equation}
\begin{tikzpicture}
\label{Basic_seq_2}
\node at (-11.75,0){$0$};
\draw[thin,->](-11.5,0)--(-11,0)node[anchor=west]{$\bar{I_C^\ast}$};
\draw[thin,->](-10.3,0)--(-9.8,0)node[anchor=west]{$\Lambda$};
\draw[thin,->](-9.3,0)--(-8.8,0)node[anchor=west]{$\Lambda$};
\draw[thin,->](-8.3,0)--(-7.8,0)node[anchor=west]{$\overline{(x-1)I_C}$};
\draw[thin,->](-5.75,0)--(-5.25,0)node[anchor=west]{$0$.};
\node at (-9,1.1){$L$};
\draw[thin,->](-9.4,.3)--(-9.2,.8);
\draw[thin,->](-8.9,.8)--(-8.7,.3);
\end{tikzpicture}
\end{equation}
Finally, recall that we constructed the exact sequence $0 \to K \to \Lambda \to \bar{I_C^\ast} \to 0$ in Section \ref{sec:K_as_the_id}. For a final time we apply $-\otimes\bar{I_C}$ and make use of Propositions \ref{Ri_otimes_K__iso_Ri} and \ref{R(1)_otimes_R(2)_is_K_oplus_free}, and Johnson's `destabilization theorem' to yield $0 \to \bar{I_C} \to \Lambda \to K \to 0$. Splicing together the previous two exact sequences yields
\begin{equation}
\begin{tikzpicture}
\label{Basic_seq_3}
\node at (-10.5,0){$0$};
\draw[thin,->](-10.25,0)--(-9.75,0)node[anchor=west]{$\bar{I_C}$};
\draw[thin,->](-9,0)--(-8.5,0)node[anchor=west]{$\Lambda$};
\draw[thin,->](-8.1,0)--(-7.6,0)node[anchor=west]{$\Lambda$};
\draw[thin,->](-7,0)--(-6.5,0)node[anchor=west]{$\bar{I_C^\ast}$};
\draw[thin,->](-5.75,0)--(-5.25,0)node[anchor=west]{$0$.};

\node at (-7.75,1.1){$K$};
\draw[thin,->](-8.1,.3)--(-7.9,.8);
\draw[thin,->](-7.7,.8)--(-7.5,.3);
\end{tikzpicture}
\end{equation}
The result now follows by splicing (\ref{Basic_seq_1}), (\ref{Basic_seq_2}) and (\ref{Basic_seq_3}).
\end{proof}

We conclude with a discussion of when $L\sim Y^\ast$. In (74.10) of \cite{Johnson2021} the exact sequence $0 \to R(3) \to P(2) \to \Lambda \to R(2) \to 0$ was constructed in which $P(2)$ is projective of rank 1. In Theorem 104.3, it was then shown that if $Inj(p,\,3)$ holds then $[P(2)]\in\mathcal{LF}(\Lambda)$. Furthermore, in this case (cf. Proposition 103.1), it follows that we have an exact sequence $0\to R(3)\to\Lambda\to K(2)\to 0$. Of course, we have shown that $K(2)\sim Y^\ast$ and we have already constructed the exact sequence $0\to \bar{I_C^\ast} \to \Lambda \to L \to 0$. It now follows from the dual form of Schanuel's Lemma that $L\sim K(2)\sim Y^\ast$. We therefore have the following:
\begin{equation}
\text{If }Inj(p,3)\text{ holds, then }L\sim Y^\ast.
\end{equation}

Recall, Johnson showed that the $\mathcal{D}(2)$ property holds when $Inj(p,3)$ holds. The issue is the difficulty in checking this property. Our results provide an alternative method that is significantly more straightforward to compute. We therefore conclude by describing a concrete procedure to verify if $L\sim Y^\ast$. From Proposition \ref{L_Y_then_D2}, this then demonstrates the $\mathcal{D}(2)$-property.

Since both $L$ and $Y$ have been constructed explicitly as quotient lattices, this condition can be investigated using integral representations of the generators $x$ and $y$ of $G(p,3)$. The procedure is as follows:
\begin{enumerate}
\item Using the quotient basis implicit in Section \ref{sec:theorem_B_part_1}, construct a $\mathbb{Z}$-basis $\mathcal{B}_{L}$ of $L$. With respect to this basis, calculate the matrices $A_L=\rho_L(x)$, $B_L=\rho_L(y)$.
\item Recall the $\tau$-invariant sublattice $\mathcal{P}\subseteq E$ defined in Section~\ref{sec:theorem_B_part_3}. Choose a $\mathbb{Z}[C_3]$-basis $\{f_1,\ldots,f_{d-1}\}$ of $\mathcal{P}$. Equivalently, $\mathcal{B}_{\mathcal{P}} = \bigl\{ f_i,\tau(f_i),\tau^2(f_i) \ \big|\ 1\leq i\leq d-1 \bigr\}$ is a $\mathbb{Z}$-basis of $\mathcal{P}$. Choose elements $g_1,g_2\in E$ such that $\mathcal{B}_{E} = \mathcal{B}_{\mathcal{P}}\cup\{g_1,g_2\}$ is a $\mathbb{Z}$-basis of $E$. Notice that $\left|\mathcal{B}_{E}\right| = p-2 = rank_{\mathbb{Z}}(E)$, as required.
\item Use $\mathcal{B}_{E}$ in the quotient construction of $Y$ from Section \ref{sec:theorem_B_part_3}. This produces a quotient basis $\mathcal{B}_{Y}$ and hence integral matrices $A_Y=\rho_Y(x)$ and $B_Y=\rho_Y(y)$. The corresponding representation of the dual lattice $Y^{\ast}$ is given by $\rho_{Y^{\ast}}(g) = \rho_Y(g^{-1})^{\mathsf{T}}$ for every $g\in G(p,3)$. In particular, since $A_Y$ and $B_Y$ are unimodular, $A_{Y^{\ast}} = \rho_{Y^{\ast}}(x) = A_Y^{-\mathsf{T}}$ and $B_{Y^{\ast}} = \rho_{Y^{\ast}}(y) = B_Y^{-\mathsf{T}}$.
\item Determine the integral intertwiners between the representations $\rho_L$ and $\rho_{Y^{\ast}}$. Thus, solve the simultaneous equations
\[
XA_L=A_{Y^{\ast}}X, \qquad XB_L=B_{Y^{\ast}}X 
\] 
for $X\in\operatorname{Mat}_{2p+1}(\mathbb{Z})$. If this system admits a solution satisfying $\det(X)=\pm1$, then $X\in GL_{2p+1}(\mathbb{Z})$ and defines an isomorphism of $\Lambda$-lattices $L\cong Y^{\ast}$.

If the direct test is unsuccessful, apply the same procedure to the stabilised lattices $L\oplus\Lambda$ and $Y^\ast\oplus\Lambda$. If $\rho_{\Lambda}$ denotes the regular integral representation of $G(p,3)$, one seeks a matrix $\tilde{X}\in GL_{5p+1}(\mathbb{Z})$ which simultaneously intertwines $\rho_L\oplus\rho_{\Lambda}$ and $\rho_{Y^{\ast}}\oplus\rho_{\Lambda}$. Such a matrix establishes $L\oplus\Lambda \cong Y^{\ast}\oplus\Lambda$ and hence $L\sim Y^{\ast}$. Thus, failure of the direct test $L\cong Y^{\ast}$ does not by itself show that the required stable isomorphism fails.
\end{enumerate}

As a final comment, we recall that it is known that Johnson's condition $Inj(p,3)$ fails when $p=229$ (it holds when $p\leq 227$). Since $Inj(p,3)\:\Rightarrow\:L\sim Y^\ast$, the first genuine case of interest would therefore be $p=229$. An interesting, if computationally monstrous, question is therefore posed:
\medskip

\noindent\textbf{Question:} Does $L\sim Y^\ast$ hold when $p=229$?
\medskip

\noindent At present, the authors are unaware of an answer to this question.

\section*{Statements and Declarations}
\noindent\textbf{Funding.}
This work was supported in part by a research grant from the
London Mathematical Society.

\medskip
\noindent\textbf{Competing interests.}
The authors have no relevant financial or non-financial interests
to disclose.

\medskip
\noindent\textbf{Data availability.}
No datasets were generated or analysed during the current study.

\medskip
\noindent\textbf{Author contributions.}
The first-named author worked on all sections of this paper. The second-named author's contributions lie in Section \ref{sec:theorem_B_part_3}.

\bibliography{Syzygies_ref}

@book{CurtisReiner1990,
   author = {C W Curtis and I Reiner},
   isbn = {0471523674},
   pages = {848},
   publisher = {Wiley-Interscience, New York},
   title = {Methods of Representation Theory, Volume 1},
   url = {https://books.google.co.uk/books/about/Methods_of_Representation_Theory.html?id=HWm_QgAACAAJ&pgis=1},
   year = {1990}
}

@article{DyerSieradski1973,
   author = {M N Dyer and A J Sieradski},
   doi = {10.1007/BF02566109},
   issn = {00102571},
   issue = {1},
   journal = {Commentarii Mathematici Helvetici},
   pages = {31-44},
   title = {Trees of homotopy types of two-dimensional CW-complexes},
   volume = {48},
   year = {1973}
}

@book{Johnson2003,
   author = {F E A Johnson},
   isbn = {0521537495},
   pages = {267},
   publisher = {Cambridge University Press, Cambridge},
   title = {Stable Modules and the D(2)-Problem},
   url = {https://books.google.com/books?id=7NBl367UTGcC&pgis=1},
   year = {2003}
}

@article{Johnson2015,
   author = {F E A Johnson},
   journal = {Communications in Algebra},
   pages = {2034-2047},
   title = {Syzygies and diagonal resolutions for dihedral groups},
   volume = {44},
   year = {2016}
}

@book{Johnson2021,
   author = {F E A Johnson},
   isbn = {978-981-122-275-7},
   pages = {357},
   publisher = {World Scientific},
   title = {Metacyclic Groups and the D(2) Problem},
   year = {2021}
}

@article{Wall1965,
   author = {C.T.C. Wall},
   journal = {Ann. of Math.},
   pages = {56-69},
   title = {Finiteness conditions for CW Complexes},
   volume = {81},
   year = {1965}
}

@article{Cassou-Nogues1973,
   author = {P Cassou-Nogués},
   journal = {Comptes Rendus Acad Sci Paris},
   pages = {A973–A975},
   title = {Classes d'ideaux de l'algebre d'un groupe abelian},
   volume = {276},
   year = {1973}
}

@article{Evans2021,
   author = {J D P Evans},
   issue = {6},
   journal = {Comm. in Algebra},
   pages = {2606-2622},
   title = {Syzygy Modules for Dihedral Groups},
   volume = {49},
   year = {2021}
}

@phdthesis{Evans2018,
   author = {J D P Evans},
   school = {University College London},
   title = {Group algebras of metacyclic type},
   year = {2018}
}

@article{Hambleton2019,
   author = {I. Hambleton},
   doi = {10.1017/S0305004118000385},
   issn = {0305-0041},
   issue = {02},
   journal = {Mathematical Proceedings of the Cambridge Philosophical Society},
   month = {9},
   pages = {361-368},
   title = {Two remarks on Wall's D2 problem},
   volume = {167},
   year = {2019}
}

@article{EilenbergGanea1957,
   author = {S. Eilenberg and T. Ganea},
   issue = {3},
   journal = {Annals of Mathematics},
   pages = {517-518},
   title = {On the Lusternik-Schnirelmann Category of Abstract Groups},
   volume = {65},
   year = {1957}
}

@article{BestvinaBrady1997,
   author = {Mladen Bestvina and Noel Brady},
   doi = {10.1007/s002220050168},
   issn = {0020-9910},
   issue = {3},
   journal = {Inventiones Mathematicae},
   month = {8},
   pages = {445-470},
   title = {Morse theory and finiteness properties of groups},
   volume = {129},
   year = {1997}
}
\bibliographystyle{plain}

\end{document}